\documentclass[12pt]{amsart}
\usepackage{amsmath}
\usepackage{amssymb,mathtools,amsthm,latexsym,bm,mathrsfs}
\usepackage{esint}
\usepackage{hyperref}
\usepackage{cleveref}

\calclayout
\makeatletter
\newcommand*{\rom}[1]{\expandafter\@slowromancap\romannumeral #1@}

\title{Sobolev versus homogeneous Sobolev extension \rom{2}}
\author{Pekka Koskela, Riddhi Mishra  and Zheng Zhu}

\address{Pekka Koskela\\
Department of Mathematics and Statistics\\,
University of Jyv\"askyl\"a, P.O. Box 35 (MaD),
FI-40014, Jyv\"askyl\"a, Finland}
\email{\tt pekka.j.koskela@jyu.fi}

\address{Riddhi Mishra\\
Department of Mathematics and Statistics\\
University of Jyv\"askyl\"a, P.O. Box 35 (MaD),
FI-40014, Jyv\"askyl\"a, Finland}
\email{\tt riddhi.r.mishra@jyu.fi}

\address{Zheng Zhu\\
School of Mathematical Science\\
Beihang University\\
        Beijing 102206\\
        P. R. China}
\email{\tt zhzhu@buaa.edu.cn}

\usepackage{stackrel}

\usepackage{color}

\numberwithin{equation}{section}
\long\def\colred#1\endred{{\color{red}#1}}
\long\def\colgreen#1\endgreen{{\color{green}#1}}
\long\def\colmagenta#1\endmagenta{{\color{magenta}#1}}
\long\def\colblue#1\endblue{{\color{blue}#1}}
\long\def\colyellow#1\endyellow{{\color{yellow}#1}}

\usepackage{tikz}
\usetikzlibrary{matrix,arrows}

\theoremstyle{plain}
\newtheorem{thm}{Theorem}[section]
\newtheorem{lem}{Lemma}[section]

\newtheorem{defn}{Definition}[section]

\newtheorem{pro}{Proposition}[section]
\numberwithin{equation}{section}

\theoremstyle{remark}
\newtheorem{remark}[equation]{Remark}

\theoremstyle{definition}

\newtheorem*{question*}{Question}

\usepackage{graphicx}


\thanks{The first two authors have been supported  by the Academy of Finland (project No. 323960). The third author has been supported by the NSFC grant (No. 12301111) and “the Fundamental Research Funds for the Central Universities” in Beihang University and the Beijing Natural Science Foundation (No. 1242007).}

\newcounter{prob}
\def\rr{{\mathbb R}}
\def\rn{{{\rr}^n}}

\def\fz{\infty}

\def\boz{{\Omega}}

\def\bint{{\ifinner\rlap{\bf\kern.25em--}
\int\else\rlap{\bf\kern.45em--}\int\fi}\ignorespaces}

\def\bbint{{\ifinner\rlap{\bf\kern.25em--}
\hspace{0.078cm}\int\else\rlap{\bf\kern.45em--}\int\fi}\ignorespaces}

\def\diam{{\mathop\mathrm{\,diam\,}}}

\def\r{\right}
\def\lf{\left}

\def\XXint#1#2#3{{\setbox0=\hbox{$#1{#2#3}{\int}$ }
\vcenter{\hbox{$#2#3$ }}\kern-.58\wd0}}

\newcommand{\barint}{
\rule[.036in]{.12in}{.009in}\kern-.16in \displaystyle\int }

\def\vint_#1{\mathchoice%
        {\mathop{\kern 0.2em\vrule width 0.6em height 0.69678ex depth -0.58065ex
                \kern -0.8em \intop}\nolimits_{\kern -0.4em#1}}%
        {\mathop{\kern 0.1em\vrule width 0.5em height 0.69678ex depth -0.60387ex
                \kern -0.6em \intop}\nolimits_{#1}}%
        {\mathop{\kern 0.1em\vrule width 0.5em height 0.69678ex
            depth -0.60387ex
                \kern -0.6em \intop}\nolimits_{#1}}%
        {\mathop{\kern 0.1em\vrule width 0.5em height 0.69678ex depth -0.60387ex
                \kern -0.6em \intop}\nolimits_{#1}}}
\def\vintslides_#1{\mathchoice%
        {\mathop{\kern 0.1em\vrule width 0.5em height 0.697ex depth -0.581ex
                \kern -0.6em \intop}\nolimits_{\kern -0.4em#1}}%
        {\mathop{\kern 0.1em\vrule width 0.3em height 0.697ex depth -0.604ex
                \kern -0.4em \intop}\nolimits_{#1}}%
        {\mathop{\kern 0.1em\vrule width 0.3em height 0.697ex depth -0.604ex
                \kern -0.4em \intop}\nolimits_{#1}}%
        {\mathop{\kern 0.1em\vrule width 0.3em height 0.697ex depth -0.604ex
                \kern -0.4em \intop}\nolimits_{#1}}}

\begin{document}
\maketitle
\begin{abstract}
We study the relationship between Sobolev extension domains and homogeneous Sobolev extension domains. Precisely, for a certain range of exponents $p$ and $q$, we construct a $(W^{1, p}, W^{1, q})$-extension domain which is not an $(L^{1, p}, L^{1, q})$-extension domain.
\end{abstract}

\section{Introduction}
Let $\mathcal X(\boz)$ be a (semi-)normed function space defined on a domain $\boz\subset\rn$ and $\mathcal Y(\rn)$ be a (semi-)normed function space defined on $\rn$. We say that $\boz\subset\rn$ is a $(\mathcal X, \mathcal Y)$-extension domain, if there exists a bounded extension operator
\[E:\mathcal X(\boz)\to\mathcal Y(\rn)\]
in the sense that for every $u\in\mathcal X(\boz)$, there exists $E(u)\in\mathcal Y(\rn)$ with 
\[E(u)\big|_\boz\equiv u\ {\rm and}\ \|E(u)\|_{\mathcal Y}\leq C\|u\|_{\mathcal X}\]
for a positive constant $C>1$ independent of $u$.

Partial motivation for the study of Sobolev extensions comes from PDEs (see, for example, \cite{Mazya}). In \cite{calderon, stein}, Calder\'on and Stein proved that for a Lipschitz domain $\boz\subset\rn$, there always exists a bounded linear extension operator 
\[E:W^{k,p}(\boz)\to W^{k, p}(\rn),\]
for each $k\geq 1$ and all $1\leq p\leq\fz$. Here $W^{k, p}(\boz)$ is the Banach space of all $L^p$-integrable functions whose distributional derivatives up to order $k$ belong to $L^p(\boz)$. In \cite{Jones:acta}, Jones introduced the notation of $(\epsilon, \delta)$-domains and generalized this result to the class of $(\epsilon, \delta)$-domains. It was proved that for every $(\epsilon, \delta)$-domain, there always exists a bounded linear extension operator
\[E:W^{k, p}(\boz)\to W^{k,p}(\rn)\] 
for each $k\geq 1$ and all $1\leq p\leq\fz$.

By the results in \cite{HKT:JFA} from Haj\l{}asz, Koskela and Tuominen, for $1\leq p<\fz$, a $(W^{1, p}, W^{1, p})$-extension domain must be Ahlfors regular in the sense that 
\begin{equation}\label{eq:ahlfors}
|B(x, r)\cap\boz|\geq C|B(x, r)|
\end{equation}
for every $x\in\overline\boz$ and all $0<r<\min\lf\{1, \frac{1}{4}\diam(\boz)\r\}$ with a constant $C$ independent of $x, r$. This excludes lots of domains, for example outward cuspidal domains which are domains with one singular point on the boundary. Sobolev extension for outward cuspidal domains has been studied in detail by several mathematicians like Maz'ya, Poborchi, Gol'dshtein and so on, for example, see \cite{GS:1982, Mazya, Mazya1, Mazya2, Mazya3}. Outward cuspidal domains are $(W^{1, p}, W^{1, q})$-extension domains for some $q,p$ with $q<p$ depending on the orders of the cusps. Recently, Koskela and Zhu proved that for $1\leq q<n-1$ and $(n-1)q/(n-1-q)\leq p<\fz$, an arbitrary outward cuspidal domain is a $(W^{1, p}, W^{1, q})$-extension domain. See \cite{KZ:cusp}. This is a limit case of the results mentioned above. 

In PDEs, people usually only care about the integrability of derivatives, for example, see \cite{M}. Hence it is essential to study the homogeneous Sobolev space $L^{1, p}(\boz)$. Here $L^{1, p}(\boz)$ is the semi-normed space of all locally $L^1$-integrable functions whose first order distributional derivatives belong to $L^p(\boz)$. Obviously, the Sobolev space $W^{1, p}(\boz)$ is always a subspace of the homogeneous Sobolev space $L^{1, p}(\boz)$. Conversely, for sufficiently nice domains, for example for balls, the homogeneous Sobolev space $L^{1, p}(\boz)$ coincides with the Sobolev space $W^{1, p}(\boz)$ as a set of functions. More generally, this holds for $p$-Poincar\'e domains. Recall that a bounded domain $\boz\subset\rn$ is called a $p$-Poincar\'e domain for $1\leq p <\infty$, if the global $(p, p)$-Poincar\'e inequality 
\begin{equation}\label{eq:p-poin}
\int_\boz|u(x)-u_\boz|^pdx\leq C\int_\boz|\nabla u(x)|^pdx
\end{equation}
holds for every $u\in W^{1, p}(\boz)$. For studies on $(L^{1,p},L^{1,p})$-extension domains, see e.g. \cite{vgl}, \cite{Jones:acta}.

Based on the above, it is natural to inquire for the relationship between $(W^{1,p},W^{1,q})$ and $(L^{1,p}, L^{1,q})$ extensions. In \cite[Section 4]{HKjam} it was proved that for $1<p<\fz$, a bounded domain is a $(W^{1, p}, W^{1, p})$-extension domain if and only if it is a $(L^{1, p}, L^{1, p})$-extension domain.

Towards the case $q<p$, let us define
\begin{equation}\label{eq:dual}
q^\star:=\begin{cases}
\frac{nq}{n-q},\ \ &\ {\rm if}\ 1\leq q<n,\\
\fz,\ \ &\ {\rm if}\ n\leq q<\fz.
\end{cases}
\end{equation}
{In \cite{KMZ:arxiv}, for a bounded domain $\Omega$, we showed that if $\Omega$ is an $(L^{1,p},L^{1,p})$-extension domain, then it is also a $(W^{1,p},W^{1,p})$-extension domain and the converse holds if $q\leq p<q^\star$.}
\begin{pro}\label{th:StoH}
Let $1\leq q\leq p<q^\star$. Then a bounded $(W^{1, p}, W^{1, q})$-extension domain is also an $(L^{1, p}, L^{1, q})$-extension domain. On the other hand, for $1\leq q <n$ and $q^* <p\leq n,$ there exists a bounded $(W^{1,p},W^{1,q})$-extension domain, which is not an $(L^{1,p},L^{1,q})$-extension domain. 
\end{pro}

Towards the missing range $p>n$, let us analyze the positive result in more details. A detailed inspection on the proof of the first part of Proposition \ref{th:StoH} gives the following conclusion. 
\begin{thm}\label{th:poincare}
Let $1\leq q\leq p<\fz$. If a $p$-Poincar\'e domain $\boz$ is a $(W^{1, p}, W^{1, q})$-extension domain, then it is also an $(L^{1, p}, L^{1, q})$-extension domain.
\end{thm}

Let $\psi:[0, \fz)\to[0, \fz)$ be a left-continuous and nondecreasing function with $\psi(0)=0$ and $\psi(1)=1$. Define the corresponding outward cuspidal domain $\boz_\psi$ by setting 
\begin{equation}\label{eq:cusp}
\boz_\psi:=\lf\{(t, z)\in (0, 1)\times\rr^{n-1}: |z|<\psi(t)\r\}\cup B^n\lf((2, 0, \cdots, 0), \sqrt{2}\r).
\end{equation}
The function $\psi$ is called the model function for the outward cuspidal domain $\boz_\psi$. The following picture explains the geometric shape of an outward cuspidal domain. 
\begin{figure}[htbp]
\centering
\includegraphics[width=0.4\textwidth]
{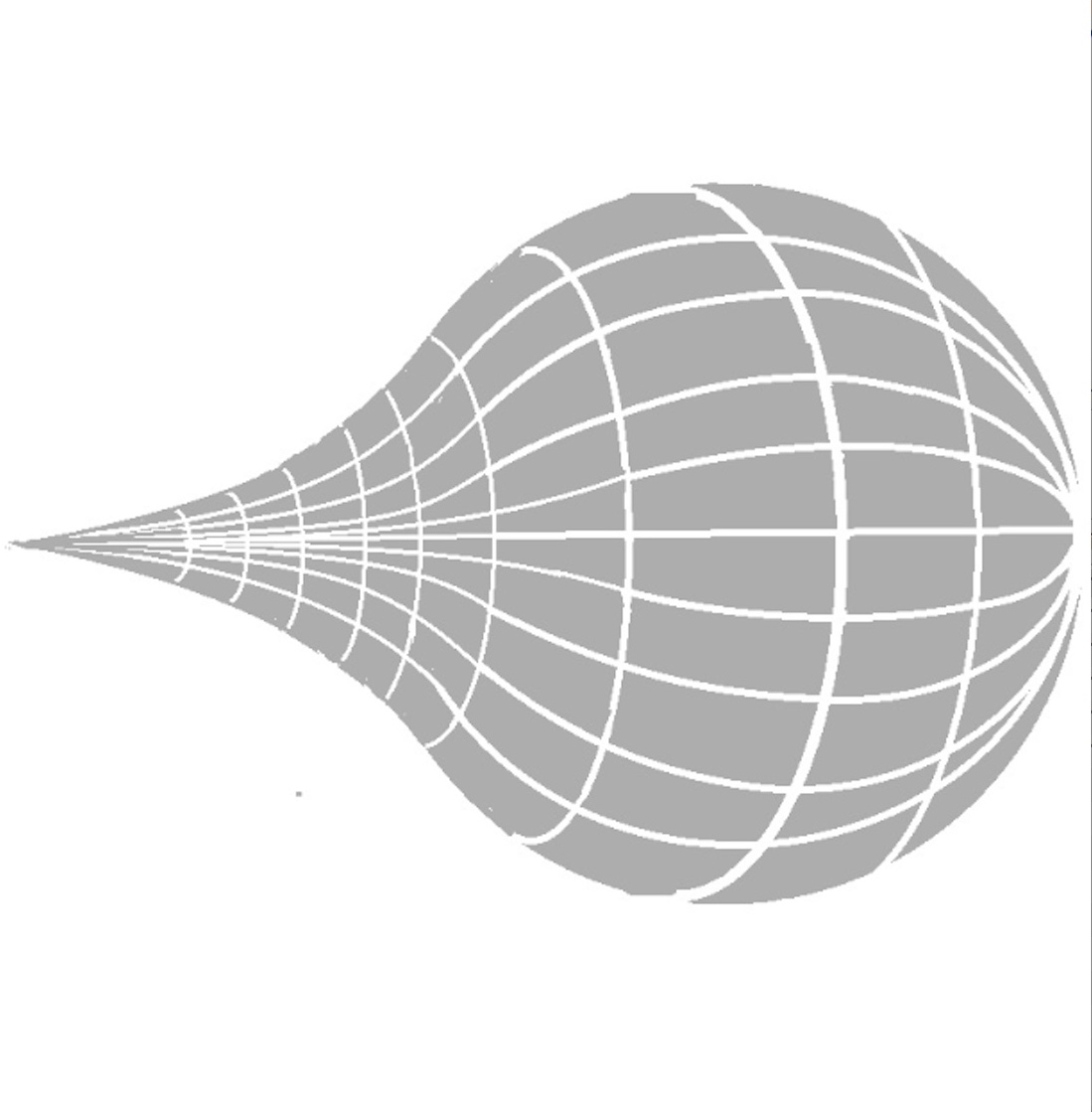}\label{fig:cusp}
\caption{The outward cuspidal domain}
\end{figure} 

In a series of papers \cite{Mazya1, Mazya2, Mazya3, Mazya5}, Maz'ya and Poborchi found sharp conditions on the cuspidal function $\psi$ to guarantee that the outward cuspidal domain $\boz_\psi$ is a $(W^{1, p}, W^{1, q})$-extension domain for $1\leq q<p<\fz$. In the paper \cite{KZ:cusp} it was proved that for an arbitrary model function $\psi$, the corresponding outward cuspidal domain $\boz_\psi$ is a $(W^{1, p}, W^{1, q})$-extension domain for every $1\leq q<n-1$ and $(n-1)q/(n-1-q)\leq p<\fz$. Furthermore, in the paper \cite{EKMZ:imrn}, it was proved that for an arbitrary outward cuspidal domain $\boz_\psi$, the Sobolev space $W^{1, p}(\boz_\psi)$ is isometrically equivalent to the Haj\l{}asz-Sobolev space $M^{1, p}(\boz_\psi)$ for every $1<p<\fz$. Also see \cite{Zhu1} for a more general result. Via this result, we can conclude that $\boz_\psi$ is a $p$-Poincar\'e domain for every $1<p<\fz$. Hence, for outward cuspidal domains, the two different Sobolev extension properties are equivalent. 
\begin{thm}\label{th:cusp}
Let $\boz_\psi\subset\rn$ be an outward cuspidal domain. Then it is a $p$-Poincar\'e domain for $1<p<\fz$. It is a $(W^{1, p}, W^{1, q})$-extension domain for some $1\leq q<p<\fz$, if and only if it is also an $(L^{1,p}, L^{1, q})$-extension domain.
\end{thm}

Our final result addresses the case $p>n$ which is not covered by the second part of the Proposition \ref{th:StoH}. Notice that the condition $q^\star <p\leq n$ in Proposition \ref{th:StoH} requires that $1\leq q<\frac{n}{2}$ and that no counterexamples can exist by the first part of the Proposition when $q \geq n$. Theorem \ref{th:n-1} below will provide us with negative examples when $1\leq q <n-1$, but the case where $n-1\leq q<n$ remains open. By Theorem \ref{th:cusp}, outward cusps do not give examples for the missing case $n-1\leq q<n$. However, the construction for Theorem \ref{th:n-1} is based on a somewhat related non-smoothness. Indeed, one relies on outward ``mushrooms'', based on cylindrical stems. The idea behind constructions of this type goes back to Nikod\'ym. The construction for the example referred to in Theorem \ref{th:poincare} is also of Nikod\'ym-type. However, the proof of the extension property heavily relies on the assumption that $p<n$. We prove the extension property of our example via ideas from \cite{KUZ:JFA} instead of using ideas from \cite{KMZ:arxiv}. Analysis of our argument shows that the restriction $p> \frac{q(n-1)}{(n-1-q)}$ cannot be avoided, see Section \ref{difdomain}. This suggest that perhaps no counterexamples for the coincidence exist when $q\geq n-1$.

\begin{thm}\label{th:n-1}
Let $n\geq 3$, $1\leq q<n-1$ and $(n-1)q/(n-1-q)<p<\fz$. Then there exists a bounded domain $\boz\subset\rn$ which is a $(W^{1, p}, W^{1, q})$-extension domain but not an $(L^{1, p}, L^{1, q})$-extension domain.
\end{thm}

\section{Preliminaries}
We refer to generic positive constants by $C$. These constants may change even in a single string of estimates. The dependence of the constant on parameters $\alpha, \beta,\cdots$ is expressed by the notation $C = C(\alpha, \beta, \cdots)$ if needed.
First, let us give the definition of Sobolev spaces and homogeneous Sobolev spaces.
\begin{defn}\label{de:sobolev}
For $u\in L^1_{\rm loc}(\boz)$, we say $u$ belongs to the homogeneous Sobolev space $L^{1, p}(\boz)$ for $1\leq p\leq\fz$ if $u$ is weakly differentiable and its weak (distributional) derivative $\nabla u$ belongs to $L^p(\boz; \rn)$. The homogeneous Sobolev space $L^{1, p}(\boz)$ is equipped with the semi-norm
\[\|u\|_{L^{1, p}(\boz)}:=\begin{cases}
\lf(\int_\boz\lf|\nabla u(x)\r|^pdx\r)^{\frac{1}{p}}, &\ {\rm for}\ 1\leq p<\fz,\\
{\rm esssup}\lf|\nabla u\r|, &\ {\rm for}\ p=\fz.
\end{cases}\] 
Furthermore, if $u\in L^{1,p}(\Omega)$ is also $L^p$-integrable, then we say $u$ belongs to the Sobolev space $W^{1 ,p}(\boz)$. The Sobolev space $W^{1, p}(\boz)$ is equipped with the norm
\[\|u\|_{W^{1, p}(\boz)}:=\begin{cases}
\lf(\int_\boz\lf|u(x)\r|^p+\lf|\nabla u(x)\r|^pdx\r)^{\frac{1}{p}}, &\ {\rm for}\ 1\leq p<\fz,\\
{\rm esssup}\lf(\lf|u\r|+\lf|\nabla u\r|\r), &\ {\rm for}\ p=\fz.
\end{cases}\] 
\end{defn}

By the results due to Calder\'on \cite{calderon} and Stein \cite{stein}, Lipschitz domains are $(W^{k, p}, W^{k, p})$-extension domains for arbitrary $k\in\mathbb N$ and $1\leq p\leq\fz$. Later, Jones generalized this result to $(\epsilon, \delta)$-domains. Every Lipschitz domain is an $(\epsilon, \delta)$-domain for some $\epsilon, \delta>0$. However, the class of $(\epsilon, \delta)$-domains is strictly larger than the class of Lipschitz domains. As we will use this result later, for the convenience of readers, we present the definition of $(\epsilon, \delta)$-domains here.
\begin{defn}\label{de:ED}
Let $0<\epsilon<1$ and $\delta>0$. We say that a domain $\boz\subset\rn$ is an $(\epsilon, \delta)$-domain, if for every $x, y\in\boz$ with $|x-y|<\delta$, there is a rectifiable curve $\gamma\subset\boz$ joining $x, y$ and satisfies the following two inequalities that 
\begin{equation}\label{eq:uni1}
{\rm length}(\gamma)\leq\frac{1}{\epsilon}|x-y|
\end{equation}
and
\begin{equation}\label{eq:uni2}
d(z, \partial\boz)\geq\frac{\epsilon|x-z||y-z|}{|x-y|}\ {\rm for\ all}\ z\ {\rm on}\ \gamma.
\end{equation}
\end{defn}

By a classical mollification argument, for arbitrary domain $\boz$, we always have that $C^\fz(\boz)\cap W^{1, p}(\boz)$ is always dense in $W^{1, p}(\boz)$ for every $1\leq p<\fz$. See \cite[Theorem 4.2]{Evans:book}. However, we cannot have that $C^\fz(\overline\boz)\cap W^{1, p}(\boz)$ is always dense in $W^{1, p}(\boz)$, unless the domain $\boz$ is sufficiently nice. For the following definition of segment condition, see e.g. the book \cite{AandF:book}.
\begin{defn}\label{de:segment}
Let $\boz\subset\rn$ be a domain. We say it satisfies the segment condition if for every $x\in\partial\boz$, there exists a neighborhood $U_x$ of $x$ and a nonzero vector $y_x$ such that if $z\in\overline\boz\cap U_x$, then $z+ty_x\in\boz$ for $0<t<1$.
\end{defn}
The following lemma tells us that the Sobolev functions defined on a domain with segment condition can be approximated by functions smooth up to the boundary.
\begin{lem}\label{le:smooth}
Let $\boz\subset\rn$ be a domain which satisfies the segment condition. Then for every $1\leq p<\fz$, the subspace $C^\fz(\overline\boz)\cap W^{1, p}(\boz)$ is dense in $W^{1, p}(\boz)$.
\end{lem}



\section{Proof of Theorem \ref{th:poincare}}
In this section, we prove Theorem \ref{th:poincare}. 
\begin{proof}[Proof of Theorem \ref{th:poincare}]
 Let $\boz\subset\rn$ be a bounded $p$-Poincar\'e domain and a $(W^{1, p}, W^{1, q})$-extension domain. By \cite{HKjam}, we know that $W^{1, p}(\Omega)$ coincides with $L^{1, p}(\Omega)$ as a set of functions. So, for $u\in L^{1,p}(\Omega),$ we have $(u-u_{\Omega})\in W^{1,p}(\Omega).$ Let $B$ be a ball with $\Omega \subset B$. We define a function on $B$ by 
$$T(u):=E(u-u_\boz)+u_\boz.$$ 
Then we have $T(u)\in W^{1, q}(B)$ with $T(u)\big|_B\equiv u$. By the fact $\boz$ is a $(W^{1, p}, W^{1, q})$-extension domain and by the $(p,p)$-Poincar\'e inequality, we have 
\begin{eqnarray}
\|\nabla T(u)\|_{L^q(B)}=\|\nabla E(u-u_\boz)\|_{L^q(B)}\leq \|u-u_{\boz}\|_{W^{1, p}(\boz)}\leq C\|\nabla u\|_{L^p(\boz)}.
\end{eqnarray}
This shows that $\boz$ is a $(L^{1, p}, L^{1, q})$-extension domain, since $B$ is such a domain.
\end{proof}
 \section{Proof of Theorem \ref{th:cusp}}
 \begin{proof}[Proof of Theorem \ref{th:cusp}]
 
    Let $\boz_\psi\subset\rn$ be an outward cuspidal domain. At first we show that $\boz_\psi$ is a $p$- Poincar\'e domain. Fix $u\in C^{\infty}(\boz_\psi)\cap W^{1,p}(\boz_\psi)$. From \cite[Lemma 4.1]{Zhu1}, there exists a non-negative function $g_u\in L^{p}(\boz_\psi)$ such that for almost every $x ,y \in \boz_\psi$, we have
    \begin{equation}\label{eq4.1}
        |u(x)-u(y)|\leq C |x-y|(g_{u}(x)+ g_{u}(y))
    \end{equation}
    and 
    \begin{equation}\label{eq4.2}
        \int_{\boz_{\psi}}g_u(z)^p dz \leq C \int_{\boz_\psi}|\nabla u(z)|^pdz
    \end{equation}
    with a constant independent of $u$.\\
    By using \ref{eq4.1} and \ref{eq4.2}, we get that
    \begin{eqnarray} \label{peq1}
        \int_{\boz_\psi}|u(x)-u_{\boz_\psi}|^pdx &=& \int_{\boz_\psi}|u(x) \barint_{\boz_\psi} dy -\barint_{\boz_\psi}u(y)dy|^p dx \nonumber\\
        &= & \int_{\boz_\psi}| \barint_{\boz_\psi}u(x) dy -\barint_{\boz_\psi}u(y)dy|^p dx \nonumber\\
        &\leq & \int_{\boz_\psi}\barint_{\boz_\psi}|u(x)-u(y)|^pdydx\nonumber\\
        &\leq& \int_{\boz_\psi}\barint_{\boz_\psi} |x-y|^p(g_u(x)+g_u(y))^pdydx\nonumber\\
        &\leq & C \diam(\boz_\psi)^p \int_{\boz_\psi}|\nabla u(z)|^pdz
    \end{eqnarray}
    Hence, \ref{peq1} shows that $\boz_\psi$ is a $p$-Poincar\'e domain for every $1\leq p \leq \infty$. By using Theorem \ref{th:poincare}, we get that if $\boz_\psi$ is a $(W^{1,p},W^{1,q})$-extension domain, then it is also an $(L^{1,p},L^{1,q})$-extension domain.
 \end{proof}
\section{Proof of Theorem \ref{th:n-1}}

\subsection{The construction of the domain}\label{construction}
Let $1\leq q<n-1$ and $(n-1)q/(n-1-q)<p<\infty$ be fixed.
We use the notation $\mathscr Q(x, r)$ to denote an $(n-1)$-dimensional closed cube in the $(n-1)$-dimensional Euclidean hyperplane $\rr^{n-1}$ with the center $x\in\rr^{n-1}$ and the side-length $2r$. Let $\mathscr Q_0:=[0, 1]^{n-1}$ be the $(n-1)$-dimensional closed unit cube in the $(n-1)$-dimensional Euclidean hyperplane $\mathbb R^{n-1}$.  We divide $\mathscr Q_0$ into $2^{n-1}$ equivalent $(n-1)$-dimensional closed cubes whose interiors are pairwise disjoint. Pick one of them and denote it by $\mathscr Q_1$.  Denote by $z_1$ to be the center of  the cube $\mathscr Q_1$. So $\mathscr Q_1=\mathscr Q\lf(z_1, \frac{1}{4}\r)$. Define
$$\mathscr B_1:=B^{n-1}\lf(z_1, \frac{1}{2^2}\r)\ {\rm and}\  \mathcal B_1:={B^{n-1}}\lf(z_1, \lf(\frac{1}{4}\r)^{\frac{1}{n-1}+\frac{p}{q}}\r).$$
Define $\mathscr C_1:=\mathscr B_1\times \lf(\frac{3}{2}, 2\r)$ to be an open cylinder and $\mathcal C_1:=\mathcal B_1\times\lf[1, \frac{3}{2}\r]$ to be a cylinder which connects the unit cube $\mathcal Q_0$ and the cylinder $\mathscr C_1$.
 
Choose one of the cubes different from $\mathscr Q_1$ from the first step and denote it by $\mathscr Q'_1$. Divide $\mathscr Q'_1$ into $2^{n-1}$ equivalent $(n-1)$-dimensional closed cubes such that their interiors are pairwise disjoint. Pick one of them up and denote it by $\mathscr Q_2$. Then its edge length is $\frac{1}{4}$. Denote $z_2$ to be the center of the cube $\mathscr Q_2$. So $\mathscr Q_2=\mathscr Q\lf(z_2, \frac{1}{8}\r)$. Define 
\[\mathscr B_2:=B^{n-1}\lf(z_2, \frac{1}{2^3}\r)\ {\rm and}\ \mathcal B_2:={B^{n-1}}\lf(z_2, \lf(\frac{1}{4^2}\r)^{\frac{1}{n-1}+\frac{p}{q}}\r).\]
Define $\mathscr C_2:=\mathscr B_2\times\lf(\frac{3}{2}, 2\r)$ to be an open cylinder and $\mathcal C_2:=\mathcal B_2\times\lf[1, \frac{3}{2}\r]$ to be a cylinder connecting the cube $\mathcal Q_0$ and the cylinder $\mathscr C_2$.

We continue the construction by induction. For $k\in\mathbb N$, assume we have already constructed $(n-1)$-dimensional closed cubes $\mathscr Q_k$, $(n-1)$-dimensional balls $\mathscr B_k$ and $\mathcal B_k$ and the cylinders $\mathscr C_k$ and $\mathcal C_k$. Here $\mathscr Q_k$ is one of $2^{n-1}$ equivalent subcubes of the dyadic cube from $\mathscr Q_{k-1}$. Fix a remaining cube in $\mathscr Q_{k-1}$ and denote it by $\mathscr Q'_{k}$. Divide $\mathscr Q'_k$ into $2^{n-1}$ equivalent $(n-1)$-dimensional closed cubes such that their interiors are pairwise disjoint. Fix one of them and denote it by $\mathscr Q_{k+1}$. Its edge length is $\frac{1}{2^{k+1}}$. Set $z_{k+1}$ to be the center of $\mathscr Q_{k+1}$. Define
\[ \mathscr B_{k+1}:= B^{n-1}\lf(z_{k+1}, \frac{1}{2^{k+2}}\r)\ {\rm and}\  \mathcal B_{k+1}:={B^{n-1}}\lf(z_{k+1}, \lf(\frac{1}{4^{k+1}}\r)^{\frac{1}{n-1}+\frac{p}{q}}\r).\]
Define $\mathscr C_{k+1}:=\mathscr B_{k+1}\times(2, 3)$ and define $\mathcal C_{k+1}:=\mathcal B_{k+1}\times [1, 2]$ to be a cylinder which connects $\mathcal Q_0$ and $\mathscr C_{k+1}$. To simplify the notation, for every $1\leq k<\fz$, we use $\tilde r_k$ and $r_k$ to denote the radius of $\mathscr B_k$ and $\mathcal B_k$ respectively. That is
\begin{equation}\label{eq:radius}
\tilde r_k=\frac{1}{2^{k+1}}\ {\rm and}\ r_k=\lf(\frac{1}{4^k}\r)^{\frac{1}{n-1}+\frac{p}{q}}.
\end{equation}

In conclusion, for every $k\in\mathbb N$, we have constructed two cylinders $\mathscr C_k$ and $\mathcal C_k$ such that the cylinder $\mathcal C_k$ connects the unit cube $\mathcal Q_0$ and the cylinder $\mathscr C_k$. Then our domain $\boz_{p ,q}$ is defined by setting
\begin{equation}\label{eq:boz}
\boz_{p, q}:=\mathcal Q_0\cup\lf(\bigcup_{k=1}^\fz\mathscr C_k\r)\cup\lf(\bigcup_{k=1}^\fz\mathcal C_k\r).
\end{equation}
Given $m\in\mathbb N$, we define
\begin{equation}\label{eq:boz1}
\boz^m_{p, q}:=\mathcal Q_0\cup\lf(\bigcup_{k=1}^m\mathscr C_k\r)\cup\lf(\bigcup_{k=1}^m\mathcal C_k\r).
\end{equation}
\begin{figure}[htbp]
\centering
\includegraphics[width=0.6\textwidth]
{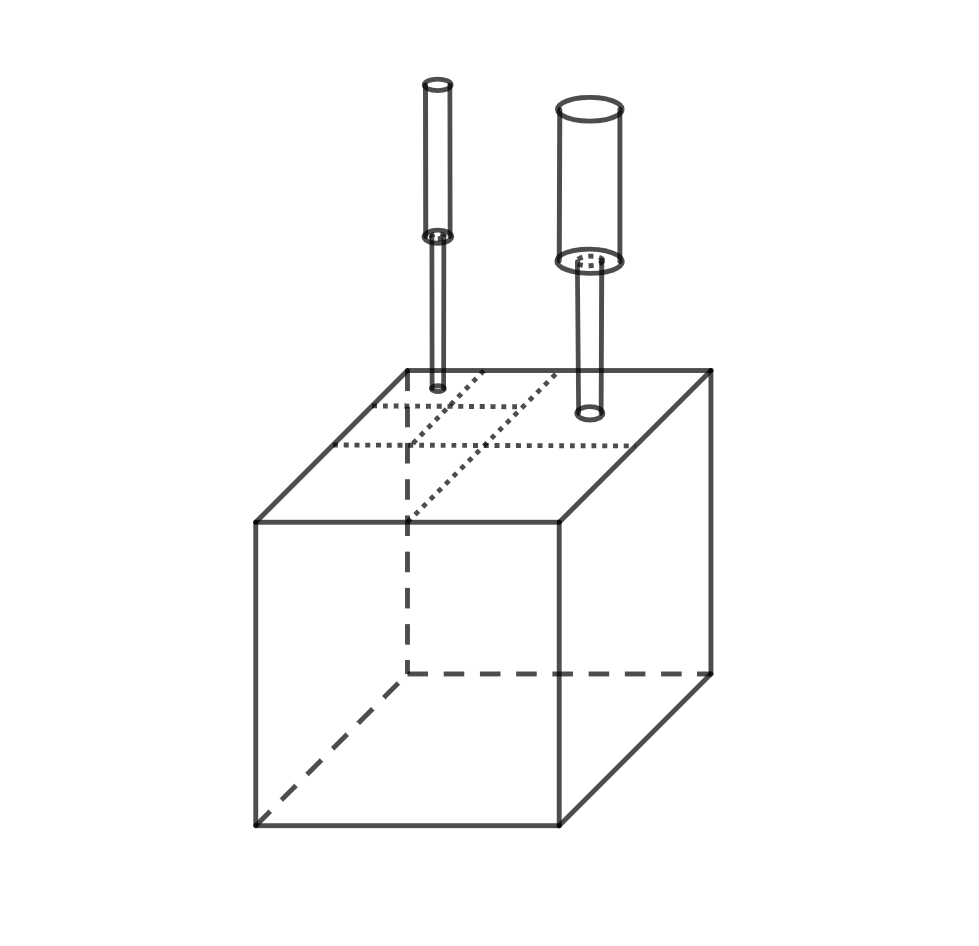}\label{fig:boz}
\caption{The domain $\boz^2_{p,q}$.}
\end{figure} 
Figure $1$ illustrates the construction of these domains. It is easy to see that the domain $\boz_{p, q}\subset\rn $ satisfies the segment condition defined in Definition \ref{de:segment}. Hence, by Lemma \ref{le:smooth}, for every $1\leq p<\fz$, $C^\fz(\overline{\boz_{p, q}})\cap W^{1, p}(\boz_{p, q})$ is dense in $W^{1, p}(\boz_{p, q})$.


Next, we show that $\boz_{p, q}$ is a $(W^{1,p},W^{1,q})$-extension domain. The essential point is that every cylinder whose size is fixed from above, is a $(W^{1,p},W^{1,q})$-extension domain for $1\leq q<n-1$ and $(n-1)q/(n-1-q)\leq p<\fz$. This observation was used in the paper \cite{KUZ:JFA} by Koskela, Ukhlov and Zhu to show that for $1\leq q<n-1$ and $(n-1)q/(n-1-q)<p\leq\fz$, there exists a bounded $(W^{1,p},W^{1,q})$-extension domain $\boz\subset\rn$ whose boundary is of positive volume. The same observation was used in the paper \cite{KZ:cusp} by Koskela and Zhu to show that every outward cuspidal domain in $\rn$ with $n\geq 3$ is always a $(W^{1,p}, W^{1,q})$-extension domain for $1\leq q<n-1$ and $(n-1)q/(n-1-q)\leq p\leq\fz$. This result is a limit case for a sequence of results about Sobolev extension property of polynomial type outward cuspidal domains by Maz'ya and Poborchi in \cite{Mazya, Mazya1, Mazya2, Mazya3}.

\subsection{Cut-off functions}
Let $\widetilde{\mathcal C}:= B^{n-1}(0,r)\times[0, 1]$ and $\mathcal C:=\overline{B^{n-1}}(0, \frac{r}{2})\times[0, 1]$. Then $\widetilde{\mathcal C}$ is a cylinder and $\mathcal C$ is a closed sub-cylinder of $\widetilde{\mathcal C}$. We define $A_\mathcal C:=\widetilde{\mathcal C}\setminus\mathcal C$. We employ the cylindrical coordinate system
\[\lf\{x=(x_1, x_2, \cdots, x_n)=(s, \theta_1, \theta_2, \cdots, \theta_{n-2}, x_n)\in\rn\r\}\]
where $\{(0, 0, \cdots, x_n):x_n\in\rr\}$ is the rotation axis and $s:=\sqrt{\sum_{i=1}^{n-1}x_i^2}$. To simplify the notation, we write $\vec{\theta}=(\theta_1, \theta_2, \cdots, \theta_{n-2})$. In this cylindrical coordinate system, we can write
\[\widetilde{\mathcal C}=\lf\{x=(s,\vec{\theta}, x_n)\in\rn: x_n\in [0, 1], s\in[0, r), \vec{\theta}\in [0, \pi)^{n-3}\times[0, 2\pi)\r\}, \]
\[{\mathcal C}=\lf\{x=(s,\vec{\theta}, x_n)\in\rn: x_n\in [0, 1], s\in\lf[0, \frac{r}{2}\r], \vec{\theta}\in [0, \pi)^{n-3}\times[0, 2\pi)\r\}\]
and
\[A_{\mathcal C}=\lf\{x=(s,\vec{\theta}, x_n)\in\rn: x_n\in [0, 1], s\in\lf(\frac{r}{2}, r\r), \vec{\theta}\in [0, \pi)^{n-3}\times[0, 2\pi)\r\}.\]
For $1<r<2$, we define two subsets $D^U_{\mathcal C}$ and $D^L_{\mathcal C}$ of the cylinder $\widetilde C$ by setting 
\[D^U_{\mathcal C}:=\lf\{x=\lf(s, \vec{\theta}, x_n\r)\in\rn: x_n\in\lf(1-\frac{r}{2},1\r), s\in\lf(\frac{r}{2}, x_n+r-1\r), \vec{\theta}\in [0, \pi)^{n-3}\times[0, 2\pi) \r\}\]
and
\[D^L_{\mathcal C}:=\lf\{x=\lf(s, \vec{\theta}, x_n\r)\in\rn: x_n\in\lf(0, \frac{r}{2}\r), s\in\lf(\frac{r}{2}, r-x_n\r), \vec{\theta}\in[0, \pi)^{n-3}\times[0, 2\pi)\r\}.\]

The next lemma gives three cut-off functions towards the construction of the desired extension operator. It is almost the same as \cite[Lemma $7.2$]{KUZ:JFA}.
\begin{lem}\label{le:cutoff}

$(1)$: There exists a cut-off function $L^i_{\mathcal C}:\overline{A_\mathcal C}\to [0, 1]$, which is continuous on $\overline{A_\mathcal C}\setminus\lf(\partial B^{n-1}\lf(0, \frac{r}{2}\r)\times\{0, 1\}\r)$, which equals $0$ both on $\lf(\overline{B}^{n-1}(0, r)\setminus B^{n-1}\lf(0, \frac{r}{2}\r)\r)\times\{0, 1\}$ and on the set $\partial B^{n-1}(0, r)\times(0, 1)$, which equals $1$ on $\partial B^{n-1}\lf(0, \frac{r}{2}\r)\times(0, 1)$ and which has the following additional properties. The function $L^i_{\mathcal C}$ is Lipschitz on $A_\mathcal C\setminus\lf(\overline{D^U_\mathcal C}\cup\overline{D^L_{\mathcal C}}\r)$ with 
\[|\nabla L^i_{\mathcal C}(x)|\leq\frac{C}{r}\ {\rm for}\ x\in A_\mathcal C\setminus\lf(\overline{D^U_\mathcal C}\cup\overline{D^L_\mathcal C}\r)\]
and $L^i_\mathcal C$ is locally Lipschitz on $D^U_\mathcal C$ with 
\[|\nabla L^i_\mathcal C(x)|\leq\frac{C}{\sqrt{\lf(s-\frac{r}{2}\r)^2+(1-x_n)^2}}\ {\rm for}\ x\in D^U_\mathcal C\]
and $L^i_\mathcal C$ is locally Lipschitz on $D^L_\mathcal C$ with 
\[|\nabla L^i_\mathcal C(x)|\leq\frac{C}{\sqrt{\lf(s-\frac{r}{2}\r)^2+x_n^2}}\ {\rm for}\ x\in D^L_{\mathcal C}.\]

$(2)$:  There exists a cut-off function $L^o_\mathcal C:\overline{A_\mathcal C}\to[0,1]$, which is continuous on $\overline{A_\mathcal C}\setminus\lf(\partial B^{n-1}\lf(0, r\r)\times\{0 ,1\}\r)$, which equals $0$ on $\partial B^{n-1}\lf(0, \frac{r}{2}\r)\times[0, 1)$, and which equals $1$ both on $\partial B^{n-1}(0, r)\times(0, 1)$ and on $\lf(B^{n-1}(0, r)\setminus\overline{B}^{n-1}\lf(0, \frac{r}{2}\r)\r)\times\{0, 1\}$, and which has the additional following properties. The function $L^o_\mathcal C$ is Lipschitz on $A_\mathcal C\setminus\lf(\overline{D^U_\mathcal C}\cup\overline{D^L_\mathcal C}\r)$  with 
\[|\nabla L^o_{\mathcal C}(x)|\leq\frac{C}{r}\ {\rm for}\ x\in A_\mathcal C\setminus\lf(\overline{D^U_\mathcal C}\cup\overline{D^L_\mathcal C}\r)\]
and $L^o_\mathcal C$ is locally Lipschitz on $D^U_\mathcal C$ with 
\[|\nabla L^o_\mathcal C(x)|\leq\frac{C}{\sqrt{\lf(s-\frac{r}{2}\r)^2+(1-x_n)^2}}\ {\rm for}\ x\in D^U_\mathcal C\]
and $L^o_\mathcal C$ is locally Lipschitz on $D^L_\mathcal C$ with 
\[|\nabla L^o_\mathcal C(x)|\leq\frac{C}{\sqrt{\lf(s-\frac{r}{2}\r)^2+x_n^2}}\ {\rm for}\ x\in D^L_{\mathcal C}.\]
\end{lem}
\begin{proof}

$(1):$ We define the cut-off function $L^i_\mathcal C$ on $\overline{A_\mathcal C}$ with respect to the cylindrical coordinate system $\lf\{x=\lf(s, \vec{\theta}, x_n\r)\in \rn\r\}$ by setting 
\begin{equation}\label{eq:cutoff1}
L^i_\mathcal C(x)=\begin{cases}
\frac{-2}{r}s+2,\ \ &\ {\rm if}\ x\in\overline{A_\mathcal C}\setminus\lf(\overline{D^U_\mathcal C}\cup\overline{D^L_\mathcal C}\r),\\
\frac{x_n}{x_n+\lf(s-\frac{r}{2}\r)}, \ \ &\ {\rm if}\ x\in \overline{D^L_\mathcal C}\setminus\partial B^{n-1}\lf(0, \frac{r}{2}\r)\times\{0\},\\
\frac{1-x_n}{(1-x_n)+\lf(s-\frac{r}{2}\r)}, \ \ &\ {\rm if}\ x\in\overline{D^U_\mathcal C}\setminus\partial B^{n-1}\lf(0, \frac{r}{2}\r)\times\{1\},\\
0, \ \ &\ {\rm if}\ x\in\partial B^{n-1}\lf(0, \frac{r}{2}\r)\times\{0, 1\}.
\end{cases}
\end{equation}
Then, if $x\in A_\mathcal C\setminus\lf(\overline{D^U_\mathcal C}\cup\overline{D^L_\mathcal C}\r)$, we have 
\[\frac{\partial L^i_\mathcal C(x)}{\partial\theta_1}=\cdots=\frac{\partial L^i_\mathcal C(x)}{\theta_{n-2}}=\frac{\partial L^i_\mathcal C(x)}{\partial x_n}=0\ {\rm and}\ \left|\frac{\partial L^i_\mathcal C(x)}{\partial s}\right|=\frac{2}{r}.\]
If $x\in D^L_\mathcal C$, we have 
\[\frac{\partial L^i_\mathcal C(x)}{\partial\theta_1}=\cdots=\frac{\partial L^i_\mathcal C(x)}{\partial\theta_{n-2}}=0,\]
\[\left|\frac{\partial L^i_\mathcal C(x)}{\partial s}\right|=\lf|\frac{x_n}{\lf(x_n+\lf(s-\frac{r}{2}\r)\r)^2}\r|\leq\lf|\frac{x_n+\lf(s-\frac{r}{2}\r)}{\lf(x_n+\lf(s-\frac{r}{2}\r)\r)^2}\r|\leq\frac{1}{\sqrt{\lf(s-\frac{r}{2}\r)^2+x_n^2}}\]
and
\[\lf|\frac{\partial L^i_\mathcal C(x)}{\partial x_n}\r|=\lf|\frac{\lf(s-\frac{r}{2}\r)}{\lf(x_n+\lf(s-\frac{r}{2}\r)\r)^2}\r|\leq\lf|\frac{x_n+\lf(s-\frac{r}{2}\r)}{\lf(x_n+\lf(s-\frac{r}{2}\r)\r)^2}\r|\leq\frac{1}{\sqrt{\lf(s-\frac{r}{2}\r)^2+x_n^2}}.\]
If $x\in D^U_\mathcal C$, we have 
\[\frac{\partial L^i_\mathcal C(x)}{\partial\theta_1}=\cdots=\frac{\partial L^i_\mathcal C(x)}{\partial\theta_{n-2}}=0,\]
\[\left|\frac{\partial L^i_\mathcal C(x)}{\partial s}\right|=\lf|\frac{1-x_n}{\lf(x_n+\lf(s-\frac{r}{2}\r)\r)^2}\r|\leq\lf|\frac{(1-x_n)+\lf(s-\frac{r}{2}\r)}{\lf((1-x_n)+\lf(s-\frac{r}{2}\r)\r)^2}\r|\leq\frac{1}{\sqrt{\lf(s-\frac{r}{2}\r)^2+(1-x_n)^2}}\]
and
\[\lf|\frac{\partial L^i_\mathcal C(x)}{\partial x_n}\r|=\lf|\frac{\lf(s-\frac{r}{2}\r)}{\lf((1-x_n)+\lf(s-\frac{r}{2}\r)\r)^2}\r|\leq\lf|\frac{(1-x_n)+\lf(s-\frac{r}{2}\r)}{\lf((1-x_n)+\lf(s-\frac{r}{2}\r)\r)^2}\r|\leq\frac{1}{\sqrt{\lf(s-\frac{r}{2}\r)^2+(1-x_n)^2}}.\]
Hence, we have 
\begin{equation}\label{eq:deri1}
\lf|\nabla L^i_\mathcal C(x)\r|\leq\begin{cases}
\frac{C}{r},\ \ &{\rm for}\ \ x\in A_\mathcal C\setminus\lf(\overline{D^U_\mathcal C}\cup\overline{D^L_\mathcal C}\r),\\
\frac{C}{\sqrt{\lf(s-\frac{r}{2}\r)^2+x_n^2}},\ \ &{\rm for}\ \ x\in D^L_\mathcal C,\\
\frac{C}{\sqrt{\lf(s-\frac{r}{2}\r)^2+\lf(1-x_n\r)^2}},\ \ &{\rm for}\ \ x\in D^U_\mathcal C.
\end{cases}
\end{equation}

$(2):$ We define the cut-off function $L^o_\mathcal C$ on $\overline{A_\mathcal C}$ with respect the cylindrical coordinate system $\lf\{x=\lf(s, \vec{\theta}, x_n\r)\in\rn\r\}$ by setting 
\begin{equation}\label{eq:cutoff2}
L^o_\mathcal C(x)=\begin{cases}
\frac{2}{r}s-1,\ \ &{\rm if}\ x\in\overline{A_\mathcal C}\setminus\lf(\overline{D^U_\mathcal C}\cup\overline{D^L_\mathcal C}\r),\\
\frac{s-\frac{r}{2}}{x_n+\lf(s-\frac{r}{2}\r)},\ \ &{\rm if}\ x\in\overline{D^L_\mathcal C}\setminus\partial B^{n-1}\lf(0, \frac{r}{2}\r)\times\{0\},\\
\frac{s-\frac{r}{2}}{(1-x_n)+\lf(s-\frac{r}{2}\r)},\ \ &{\rm if}\ x\in\overline{D^U_\mathcal C}\setminus\partial B^{n-1}\lf(0, \frac{r}{2}\r)\times\{1\},\\
0,\ \ &{\rm if}\ x\in\partial B^{n-1}\lf(0, \frac{r}{2}\r)\times\{0, 1\}.\end{cases}
\end{equation}
By similar computations, we have 
\begin{equation}\label{eq:deri2}
\lf|\nabla L^o_\mathcal C(x)\r|\leq\begin{cases}
\frac{C}{r},\ \ &{\rm for}\ \ x\in A_\mathcal C\setminus\lf(\overline{D^U_\mathcal C}\cup\overline{D^L_\mathcal C}\r),\\
\frac{C}{\sqrt{\lf(s-\frac{r}{2}\r)^2+x_n^2}},\ \ &{\rm for}\ \ x\in D^L_\mathcal C,\\
\frac{C}{\sqrt{\lf(s-\frac{r}{2}\r)^2+\lf(1-x_n\r)^2}},\ \ &{\rm for}\ \ x\in D^U_\mathcal C.
\end{cases}
\end{equation}
\end{proof}

\subsection{The extension operator}
Towards the construction of our extension operator, we define piston-shaped domains $P_k$ by setting
\[P_k:=2\mathcal B_k\times(0, 1)\cup\mathcal B_k\times[1, 2).\]
The class of piston-shaped domains $\{P_k\}$ are pairwise disjoint. Given cylinders $\mathscr C_k$ and  $\mathcal C_k$, to simplify the notation, we write $L_k=L_{\mathscr C_k}, L^i_k=L^i_{\mathcal C_k}$ and $L^o_k=L^o_{\mathcal C_k}$. We define cut-off functions $L$, $L^i$ and $L^o$ by setting 
\begin{equation}\label{4.8}
\mathscr L^i(x):=\sum_kL^i_k(x)\ {\rm for}\ x\in\bigcup_k\overline{\mathscr A_k},
\end{equation}
\begin{equation}\label{4.9}
L^i(x):=\sum_kL^i_k(x)\ {\rm for}\ x\in\bigcup_k\overline{\mathcal A_k}
\end{equation}
and
\begin{equation}\label{4.10}
L^o(x):=\sum_kL^o_k(x)\ {\rm for}\ x\in\bigcup_k\overline{\mathcal A_k}.
\end{equation}
On $\mathscr Q_0\times[1, 2]$, we define a cut-off function $L_1$ by setting 
\begin{equation}\label{eq:Ref1}
L_1(x):=2-x_n\ {\rm for\ every}\ x=(x_1, x_2, \cdots, x_n)\in\mathscr Q_0\times[1, 2].
\end{equation}

On $\mathscr Q_0\times(1, 2)$, we define a reflection $ R_1$ by setting 
\begin{equation}\label{eq:ref1}
 R_1(x):=\lf(x_1, x_2,\cdots, x_{n-1}, 2-x_n\r)\ {\rm for\ every}\ x=(x_1, x_2, \cdots, x_n)\in\mathscr Q_0\times(1, 2).
\end{equation}
Simple computations give the estimates
\begin{equation}\label{eq:esti1}
\frac{1}{C}\leq\lf|J_{ R_1}(x)\r|\leq C\ {\rm and}\ \lf|DR_1(x)\r|\leq C,
\end{equation}
for every $x\in\mathscr Q_0\times(1, 2)$.

For every $k\in\mathbb N$, we define two cylinders by setting 
\[\widetilde{\mathscr C_k}:=2\mathscr B_k\times (2, 3)\ {\rm and}\ \widetilde{\mathcal C_k}:=2\mathcal B_k\times[1, 2],\]
and define 
\[\mathscr A_k:=\widetilde{\mathscr C_k}\setminus\overline{\mathscr C_k}\ {\rm and}\ \mathcal A_k:=\widetilde{\mathcal C_k}\setminus\overline{\mathcal C_k}.\]
We also define 
\[\mathscr D^L_k:= D^L_{\mathscr C_k}, \mathcal D^L_k:= D^L_{\mathcal C_k}\ {\rm and}\ \mathcal D^U_{k}:=D^U_{\mathcal C_k}.\]
On the set $\bigcup_{k}\mathscr A_k$, we define a mapping $\mathscr R$ which is a reflection on every $\mathscr A_k$. With respect to the local cylindrical coordinate system on every $\mathscr A_k$, we write
\begin{equation}\label{eq:ref3}
\mathscr R(x):=\mathscr R\lf(s, \vec{\theta}, x_n\r)=\lf(-\frac{s}{2}+\frac{3}{2}\tilde r_k, \vec{\theta}, x_n\r)
\end{equation}
for $x=(s, \vec{\theta}, x_n)\in\mathscr A_k$. Simple computations give the estimate
\begin{equation}\label{eq:esti3}
\frac{1}{C}\leq\lf|J_{\mathscr R}(x)\r|\leq C\ {\rm and}\ \lf|D\mathscr R(x)\r|\leq C
\end{equation}
for every $x\in\bigcup_{k}\mathscr A_k$. On the set $\bigcup_{k}\mathcal A_k$, we define a mapping $\mathcal R$ which is a reflection on every $\mathcal A_k$. With respect to the local cylindrical coordinate system on every $\mathcal A_k$, we write
\begin{equation}\label{eq:ref4}
\mathcal R(x):=\mathcal R\lf(s, \vec{\theta}, x_n\r)=\lf(-\frac{s}{2}+\frac{3}{2}r_k, \vec{\theta}, x_n\r)
\end{equation}
for $x=(s, \vec{\theta}, x_n)\in\mathcal A_k$. Simple computations give the estimate
\begin{equation}\label{eq:esti4}
\frac{1}{C}\leq\lf|J_\mathcal R(x)\r|\leq C\ {\rm and}\ \lf|D\mathcal R(x)\r|\leq C
\end{equation}
for every $x\in\bigcup_{k}\mathcal A_k$.

Define $\mathcal C:=\mathscr Q_0\times(0, 3)$ to be a cylinder which contains the domain $\boz_{p, q}$. In order to explain the construction of our bounded linear extension operator, we need to define a subset $U_1$ of $\mathcal C$. Define 
$$U_1:=\mathscr Q_0\times[1, 2]\setminus\bigcup_k{\mathcal C_k} $$
  As we said, $\boz_{p, q}\subset\rn$ satisfies the segment condition. Hence, $C^\fz(\overline{\boz_{p, q}})\cap W^{1, p}(\boz_{p, q})$ is dense in $W^{1, p}(\boz_{p, q})$. We begin by defining our linear extension operator on the dense subspace $C^\fz(\overline{\boz_{p, q}})\cap W^{1, p}(\boz_{p, q})$ of $W^{1, p}(\boz_{p, q})$. Given $u\in C^\fz(\overline{\boz_{p, q}})\cap W^{1, p}(\boz_{p, q})$, we define the extension $E(u)$ on the rectangle $\mathcal C$ by setting 
\begin{equation}\label{eq:extension}
E(u)(x):=\begin{cases}
u(x), \ & {\rm for}\ x\in\boz_{p, q},\\
L_1(x)\lf(u\circ R_1(x)\r), \ & {\rm for}\ x\in U_1\setminus\bigcup_k\overline{\mathcal A_k},\\
\mathscr L^i(x)\lf(u\circ\mathscr R(x)\r), \ & {\rm for}\  x\in\bigcup_k\overline{\mathscr A_k},\\
L_1(x)L^o(x)\lf(u\circ R_1(x)\r)+L^i(x)\lf(u\circ\mathcal R(x)\r), \ & {\rm for}\ x\in\bigcup_k\overline{\mathcal A_k},\\
0, \ & {\rm elsewhere}.
\end{cases}
\end{equation}
We continue with the local properties of our extension operator.
\begin{lem}\label{le:extension}
Let $E$ be the extension operator defined in (\ref{eq:extension}). Then, for every $u\in C^\fz(\overline{\boz_{p, q}})\cap W^{1, p}(\boz_{p, q})$, we have $E(u)\in C(\mathcal C) \cap W^{1,q}(\mathcal{C})$ and following properties:

$(1)$: $E(u)$ is differentiable almost everywhere on $\mathring U_1\setminus\bigcup_k\overline{\mathcal A_k}$ with 
\begin{equation}\label{eq:1}
\lf|\nabla E(u)(x)\r|\leq \lf|\nabla L_1(x)\lf(u\circ R_1(x)\r)\r|+\lf|L_1(x)\nabla\lf(u\circ R_1(x)\r)\r|
\end{equation}
for almost every $x\in\mathring{U_1}\setminus\bigcup_k\overline{\mathcal A_k}$.

$(2)$: $E(u)$ is differentiable almost everywhere on $\mathscr A_k$ with 
\begin{equation}\label{eq2}
\lf|\nabla E(u)(x)\r|\leq\lf|\nabla\mathscr L^i(x)\lf(u\circ\mathscr R(x)\r)\r|+\lf|\mathscr L^i(x)\nabla\lf(u\circ\mathscr R(x)\r)\r|
\end{equation}
for almost every $x\in\mathscr A_k$.

$(3)$: $E(u)$ is differentiable almost everywhere on $\mathcal A_k$ with 
\begin{multline}\label{eq3}
\lf|\nabla E(u)(x)\r|\leq\lf|\nabla \lf(L_1(x)L^o(x)\r)\lf(u\circ R_1(x)\r)\r|+\lf|L_1(x)L^o(x)\nabla\lf(u\circ R_1(x)\r)\r|\\
+\lf|\nabla L^i(x)\lf(u\circ\mathcal R(x)\r)\r|+\lf|L^i(x)\nabla\lf(u\circ\mathcal R(x)\r)\r|
\end{multline}
for almost every $x\in\mathcal A_k$.

Moreover, with respect to the local cylindrical coordinate system $x=\lf(s, \vec{\theta}, x_n\r)$ on $\widetilde{\mathscr C_k}$, for every $1\leq q<\fz$, we have 
\begin{equation}\label{eq:norm1}
\int_{\widetilde{\mathscr C_k}}\lf|E(u)(x)\r|^qdx\leq C\int_{\mathscr C_k}\lf|u(x)\r|^qdx
\end{equation}
and 
\begin{multline}\label{eq:norm2}
\int_{\widetilde{\mathscr C_k}}\lf|\nabla E(u)(x)\r|^qdx\leq C\int_{\mathscr C_k}\lf|\nabla u(x)\r|^qdx\\
+C\int_{\mathscr D^L_k}\lf(\sqrt{\frac{1}{x_n^2+\lf(s-\frac{\tilde r_k}{2}\r)^2}}\r)^q\lf|u\circ\mathscr R(x)\r|^qdx\\
+C\int_{\mathscr D^U_k}\lf(\sqrt{\frac{1}{\lf(1-x_n\r)^2+\lf(s-\frac{\tilde r_k}{2}\r)^2}}\r)^q\lf|u\circ\mathscr R(x)\r|^qdx\\
+C\int_{\mathscr A_k\setminus\lf(\overline{\mathscr D^L_k}\cup\overline{\mathscr D^U_k}\r)}\lf(\frac{1}{\tilde r_k}\r)^q\lf|u\circ\mathscr R(x)\r|^qdx,
\end{multline}
with some uniform positive constant $C$. With respect to the local cylindrical coordinate system $x=\lf(s, \vec{\theta}, x_n\r)$ on $\widetilde{\mathcal C_k}$, for every $1\leq q<\fz$, we have 
\begin{equation}\label{eq:norm3}
\int_{\widetilde{\mathcal C_k}}\lf|E(u)(x)\r|^qdx\leq C\int_{P_k}\lf|u(x)\r|^qdx
\end{equation}
and
\begin{multline}\label{eq:norm4}
\int_{\widetilde{\mathcal C_k}}\lf|\nabla E(u)(x)\r|^qdx\leq C\int_{P_k}\lf|\nabla u(x)\r|^qdx\\
+C\int_{\mathcal D^U_k}\lf(\sqrt{\frac{1}{\lf(1-x_n\r)^2+\lf(s-\frac{r_k}{2}\r)^2}}\r)^q\lf(\lf|u\circ R_1(x)\r|^q+\lf|u\circ\mathcal R(x)\r|^q\r)dx\\
+C\int_{\mathcal D^L_k}\lf(\sqrt{\frac{1}{x_n^2+\lf(s-\frac{r_k}{2}\r)^2}}\r)^q\lf(\lf|u\circ R_1(x)\r|^q+\lf|u\circ\mathcal R(x)\r|^q\r)dx\\
+C\int_{\mathcal A_k\setminus\lf(\overline{\mathcal D^L_k}\cup\overline{\mathcal D^U_k}\r)}\lf(\frac{1}{r_k}\r)^q\lf(\lf|u\circ R_1(x)\r|^q+\lf|u\circ\mathcal R(x)\r|^q\r)dx,
\end{multline}
with some uniform positive constant $C$.
\end{lem}
\begin{proof}
For every $u\in C^\fz(\overline{\boz_{p, q}})\cap W^{1, p}(\boz_{p, q})$, by the properties of our cut-off functions and reflections and the definition of $E(u)$, it is easy to see that $E(u)$ belongs to $W^{1,q}_{loc}(\mathcal{C})$ and differentiable almost everywhere both on $\mathring{U_1}$, $\mathscr A_k$ and $\mathcal A_k$ for every $k\in\mathbb N$. The inequalities (\ref{eq:1}), (\ref{eq2}) and (\ref{eq3}) come from the chain rule. 

By the definition of $E(u)$ in (\ref{eq:extension}), on every $\widetilde{\mathscr C_k}$, we have 
\begin{equation}\label{eq4}
\int_{\widetilde{\mathscr C_k}}\lf|E(u)(x)\r|^qdx\leq\int_{\mathscr C_k}\lf|u(x)\r|^qdx+\int_{\mathscr A_k}\lf|\mathscr L^i(x)\lf(u\circ\mathscr R(x)\r)\r|^qdx.
\end{equation}
Since $0\leq\mathscr L^i(x)\leq 1$ for every $x\in\mathscr A_k$, by (\ref{eq:esti3}) and the change of variables formula, we have 
\begin{equation}\label{eq5}
\int_{\mathscr A_k}\lf|\mathscr L^i(x)\lf(u\circ\mathscr R(x)\r)\r|^qdx\leq\int_{\mathscr A_k}\lf|u\circ\mathscr R(x)\r|^qdx\leq C\int_{\mathscr C_k}\lf|u(x)\r|^qdx.
\end{equation}
By combining inequalities (\ref{eq4}) and (\ref{eq5}), we obtain the inequality (\ref{eq:norm1}). 

By inequality (\ref{eq2}), we have 
\begin{multline}\label{eq6}
\int_{\widetilde{\mathscr C_k}}\lf|\nabla E(u)(x)\r|^qdx\leq\int_{\mathscr C_k}\lf|\nabla u(x)\r|^qdx\\
+C\int_{\mathscr A_k}\lf|\nabla\mathscr L^i(x)\lf(u\circ\mathscr R(x)\r)\r|^qdx+C\int_{\mathscr A_k}\lf|\mathscr L^i(x)\nabla\lf(u\circ\mathscr R(x)\r)\r|^qdx
\end{multline}
Since $0\leq\mathscr L^i(x)\leq 1$ for every $x\in\mathscr A_k$, we have 
\begin{equation}\label{eq7}
\int_{\mathscr A_k}\lf|\mathscr L^i(x)\nabla\lf(u\circ\mathscr R(x)\r)\r|^qdx\leq\int_{\mathscr A_k}\lf|\nabla\lf(u\circ\mathscr R(x)\r)\r|^qdx.
\end{equation}
By (\ref{eq:deri1}), we have 
\begin{multline}\label{eq8}
\int_{\mathscr A_k}\lf|\nabla\mathscr L^i(x)\lf(u\circ\mathscr R(x)\r)\r|^qdx\leq\int_{\mathscr D^L_k}\lf(\sqrt{\frac{1}{x_n^2+\lf(s-\frac{\tilde r_k}{2}\r)^2}}\r)^q\lf|u\circ\mathscr R(x)\r|^qdx\\
+\int_{\mathscr D^U_k}\lf(\sqrt{\frac{1}{\lf(1-x_n\r)^2+\lf(s-\frac{\tilde r_k}{2}\r)^2}}\r)^q\lf|u\circ\mathscr R(x)\r|^qdx\\
+\int_{\mathscr A_k\setminus\lf(\overline{\mathscr D^L_k}\cup\overline{\mathscr D^U_k}\r)}\lf(\frac{1}{\tilde r_k}\r)^q\lf|u\circ\mathscr R(x)\r|^qdx.
\end{multline}
By combining inequalities (\ref{eq6}), (\ref{eq7}) and (\ref{eq8}), we obtain the inequality (\ref{eq:norm2}).

By (\ref{eq:extension}), we have 
\begin{multline}\label{eq9}
\int_{\widetilde{\mathcal C_k}}\lf|E(u)(x)\r|^qdx\leq\int_{P_k}\lf|u(x)\r|^qdx\\
+C\int_{\mathcal A_k}\lf|L_1(x)L^o(x)\lf(u\circ R_1(x)\r)+L^i(x)\lf(u\circ\mathcal R(x)\r)\r|^qdx.
\end{multline}
Since $0\leq L_1(x)L^o(x)\leq 1$ and $0\leq L^i(x)\leq 1$ on $\mathcal A_k$, the definition of reflections and the change of variables formula imply 
\begin{multline}\label{eq10}
\int_{\mathcal A_k}\lf|L_1(x)L^o(x)\lf(u\circ R_1(x)\r)+L^i(x)\lf(u\circ\mathcal R(x)\r)\r|^qdx\leq C\int_{\mathcal A_k}\lf|u\circ R_1(x)\r|^qdx\\
+C\int_{\mathcal A_k}\lf|u\circ\mathcal R(x)\r|^qdx\leq C\int_{P_k}\lf|u(x)\r|^qdx.
\end{multline}
By (\ref{eq3}), we have 
\begin{equation}\label{eq11}
\int_{\widetilde{\mathcal C_k}}\lf|\nabla E(u)(x)\r|^qdx\leq\int_{P_k}\lf|\nabla u(x)\r|^qdx+ CI_1^k+CI_2^k,
\end{equation}
where 
\[I_1^k:=\int_{\mathcal A_k}\lf|L_1(x)L^o(x)\nabla\lf(u\circ R_1(x)\r)\r|^qdx+\int_{\mathcal A_k}\lf|L^i(x)\nabla\lf(u\circ\mathcal R(x)\r)\r|^qdx\]
and
\begin{equation}\label{eq11'}
I_2^k:=\int_{\mathcal A_k}\lf|\nabla\lf(L_1(x)L^o(x)\r)\lf(u\circ R_1(x)\r)\r|^qdx+\int_{\mathcal A_k}\lf|\nabla L^i(x)\lf(u\circ\mathcal R(x)\r)\r|^qdx.
\end{equation}
Since $0\leq L_1(x)L^o(x)\leq 1$ and $0\leq L^i(x)\leq 1$ on $\mathcal A_k$, the change of variables formula implies
\begin{equation}\label{eq12}
I^k_1\leq C\int_{P_k}\lf|\nabla u(x)\r|^qdx.
\end{equation}
The chain rule gives
\begin{equation}\label{eq12'}
\lf|\nabla\lf(L_1(x)L^o(x)\r)\r|\leq\lf|L^o(x)\nabla L_1(x)\r|+\lf|L_1(x)\nabla L^o(x)\r|\leq2\lf|\nabla L^o(x)\r|,
\end{equation}
for every $x\in\mathcal A_k$. Hence, by (\ref{eq:deri2}), we have 
\begin{multline}\label{eq13}
\int_{\mathcal A_k}\lf|\nabla\lf(L_1(x)L^o(x)\r)\lf(u\circ R_1(x)\r)\r|^qdx\leq C\int_{\mathcal A_k}\lf|\nabla L^o(x)\lf(u\circ R_1(x)\r)\r|^qdx\\
\leq\int_{\mathcal D^U_k}\lf(\sqrt{\frac{1}{\lf(1-x_n\r)^2+\lf(s-\frac{r_k}{2}\r)^2}}\r)^q\lf|u\circ R_1(x)\r|^qdx\\
+\int_{\mathcal D^L_k}\lf(\sqrt{\frac{1}{x_n^2+\lf(s-\frac{r_k}{2}\r)^2}}\r)^q\lf|u\circ R_1(x)\r|^qdx\\
+\int_{\mathcal A_k\setminus\lf(\overline{\mathcal D^U_k}\cup\overline{\mathcal D^L_k}\r)}\lf(\frac{1}{r_k}\r)^q\lf|u\circ R_1(x)\r|^qdx.
\end{multline}
By (\ref{eq:deri1}), we have 
\begin{multline}\label{eq14}
\int_{\mathcal A_k}\lf|\nabla L^i(x)\lf(u\circ\mathcal R(x)\r)\r|^qdx\leq\int_{\mathcal D^U_k}\lf(\sqrt{\frac{1}{\lf(1-x_n\r)^2+\lf(s-\frac{r_k}{2}\r)^2}}\r)^q\lf|u\circ R_1(x)\r|^qdx\\
+\int_{\mathcal D^L_k}\lf(\sqrt{\frac{1}{x_n^2+\lf(s-\frac{r_k}{2}\r)^2}}\r)^q\lf|u\circ R_1(x)\r|^qdx\\
+\int_{\mathcal A_k\setminus\lf(\overline{\mathcal D^U_k}\cup\overline{\mathcal D^L_k}\r)}\lf(\frac{1}{r_k}\r)^q\lf|u\circ R_1(x)\r|^qdx.
\end{multline}
By combining inequalities (\ref{eq11})-(\ref{eq14}), we obtain the inequality (\ref{eq:norm4}).
\end{proof}

\subsection{Proof of Theorem \ref{th:n-1}}
In this section, we give a proof for Theorem \ref{th:n-1}. We divide it into two parts. First, we prove that for every $1\leq q<n-1$ and $(n-1)q/(n-1-q)<p<\fz$, $E$ defined in (\ref{eq:extension}) is a bounded linear extension operator from $W^{1, p}(\boz_{p, q})$ to $W^{1, q}(\rn)$.

\begin{thm}\label{th:Sobexten}
For $1\leq q<n-1$ and $(n-1)q/(n-1-q)<p<\fz$, the corresponding domain $\boz_{p, q}\subset\rn$ constructed in Section \ref{construction} is a $(W^{1, p}, W^{1, q})$-extension domain. 
\end{thm}
\begin{proof}
Given $u\in C^\fz(\rn)\cap W^{1, p}(\boz_{p, q})$, the function $E(u)$ defined in (\ref{eq:extension}) belongs to $W^{1,q}_{loc}(\mathcal{C})$ and differentiable almost everywhere on $\mathcal C$. By the definition of $E(u)$ in (\ref{eq:extension}), we have 
\begin{equation}\label{eq15}
\int_{\mathcal C}\lf|E(u)(x)\r|^qdx=\lf(\int_{\boz_{p, q}}+\int_{\bigcup_k\mathscr A_k}+\int_{\bigcup_{k}\mathcal A_k}+\int_{U_1\setminus\bigcup_{k}\overline{\mathcal A_k}}\r)\lf|E(u)(x)\r|^qdx.
\end{equation}
The H\"older inequality implies 
\begin{equation}\label{eq16}
\int_{\boz_{p, q}}\lf|E(u)(x)\r|^q dx\leq C\lf(\int_{\boz_{p,q}}\lf|u(x)\r|^pdx\r)^{\frac{q}{p}}.
\end{equation}
Since the collection $\{\mathscr C_k\}$ is pairwise disjoint, the inequality (\ref{eq:norm1}) and the H\"older inequality imply 
\begin{eqnarray}\label{eq17}
\int_{\bigcup_k\mathscr A_k}\lf|E(u)(x)\r|^qdx&\leq&\int_{\bigcup_k\widetilde{\mathscr C_k}}\lf|E(u)(x)\r|^qdx\\
&\leq&C\int_{\bigcup_k\mathscr C_k}\lf|u(x)\r|^qdx\leq C\lf(\int_{\boz_{p, q}}\lf|u(x)\r|^pdx\r)^{\frac{q}{p}}.\nonumber
\end{eqnarray}
Since the collection $\{P_k\}$ is pairwise disjoint, the inequality (\ref{eq:norm3}) and the H\"older inequality imply 
\begin{eqnarray}\label{eq18}
\int_{\bigcup_k\mathcal A_k}\lf|E(u)(x)\r|^qdx&\leq&\int_{\bigcup_k\widetilde{\mathcal C_k}}\lf|E(u)(x)\r|^qdx\\
&\leq&C\int_{\bigcup_k\mathcal C_k}\lf|u(x)\r|^qdx\leq C\lf(\int_{\boz_{p, q}}\lf|u(x)\r|^pdx\r)^{\frac{q}{p}}.\nonumber
\end{eqnarray}
Since $0\leq L_1(x)\leq 1$ on $U_1\setminus\bigcup_k\overline{\mathcal A_k}$, by (\ref{eq:esti1}), the change of variables formula and the H\"older inequality, we have 
\begin{eqnarray}\label{eq19}
\int_{U_1\setminus\overline{\mathcal A_k}}\lf|E(u)(x)\r|^qdx&\leq &\int_{U_1\setminus\overline{\mathcal A_k}}\lf|u\circ R_1(x)\r|^qdx\\                    &\leq&C\int_{R_1\lf(U_1\setminus\overline{\mathcal A_k}\r)}\lf|u(x)\r|^qdx\leq C\lf(\int_{\boz_{p, q}}\lf|u(x)\r|^pdx\r)^{\frac{q}{p}}.\nonumber  
\end{eqnarray}
Consequently, by combining inequalities (\ref{eq15}) to (\ref{eq19}), we obtain 
\begin{equation}\label{eq20}
\lf(\int_{\mathcal C}\lf|E(u)(x)\r|^qdx\r)^{\frac{1}{q}}\leq C\lf(\int_{\boz_{p, q}}\lf|u(x)\r|^pdx\r)^{\frac{1}{p}},
\end{equation}
where the constant $C$ is independent of $u$.

By the definition of $E(u)$ in (\ref{eq:extension}), we have 
\begin{equation}\label{eq21}
\int_{\mathcal C}\lf|\nabla E(u)(x)\r|^q dx=\lf(\int_{\boz_{p, q}}+\int_{\bigcup_k\mathscr A_k}+\int_{\bigcup_{k}\mathcal A_k}+\int_{\mathring U_1\setminus\bigcup_{k}\overline{\mathcal A_k}}\r)\lf|\nabla E(u)(x)\r|^q dx.
\end{equation}
Similarly, the H\"older inequality implies 
\begin{equation}\label{eq22}
\int_{\boz_{p, q}}\lf|\nabla E(u)(x)\r|^qdx\leq C\lf(\int_{\boz_{p, q}}\lf|\nabla u(x)\r|^pdx\r)^{\frac{q}{p}}.
\end{equation}
By (\ref{eq:1}), we have 
\begin{eqnarray}\label{eq23}
\int_{\mathring{U_1}\setminus\bigcup_k\overline{\mathcal A_k}}\lf|\nabla E(u)(x)\r|^qdx&\leq&C\int_{\mathring{U_1}\setminus\bigcup_k\overline{\mathcal A_k}}\lf|\nabla L_1(x)\lf(u\circ R_1(x)\r)\r|^qdx\\
        & &+C\int_{\mathring{U_1}\setminus\bigcup_k\overline{\mathcal A_k}}\lf|L_1(x)\nabla\lf(u\circ R_1(x)\r)\r|^qdx.\nonumber
\end{eqnarray}
Obviously, $|\nabla L_1(x)|=1$ for every $x\in\mathring{U_1}\setminus\bigcup_k\overline{\mathcal A_k}$. By (\ref{eq:esti1}), the change of variables formula and the H\"older inequality, we have 
\begin{eqnarray}\label{eq24}
\int_{\mathring{U_1}\setminus\bigcup_k\overline{\mathcal A_k}}\lf|\nabla L_1(x)\lf(u\circ R_1(x)\r)\r|^qdx&\leq&C\int_{\mathring{U_1}\setminus\bigcup_k\overline{\mathcal A_k}}\lf|u\circ R_1(x)\r|^qdx\\
        &\leq&C\int_{R_1\lf(\mathring U_1\setminus\bigcup_k\overline{\mathcal A_k}\r)}\lf|u(x)\r|^qdx\nonumber\\
        &\leq&C\lf(\int_{\boz_{p, q}}\lf|u(x)\r|^pdx\r)^{\frac{q}{p}}.\nonumber
\end{eqnarray}
Since $0\leq L_1(x)\leq 1$ for every $x\in\mathring{U_1}\setminus\bigcup_k\overline{\mathcal A_k}$, also by (\ref{eq:esti1}), the change of variables formula and the H\"older inequality, we have 
\begin{eqnarray}\label{eq25}
\int_{\mathring{U_1}\setminus\bigcup_k\overline{\mathcal A_k}}\lf|L_1(x)\nabla\lf(u\circ R_1(x)\r)\r|^qdx&\leq&\int_{\mathring{U_1}\setminus\bigcup_k\overline{\mathcal A_k}}\lf|\nabla\lf(u\circ R_1(x)\r)\r|^q\\
        &\leq&C\int_{R_1\lf(\mathring{U_1}\setminus\bigcup_k\overline{\mathcal A_k}\r)}\lf|\nabla u(x)\r|^qdx\nonumber\\
        &\leq&C\lf(\int_{\boz_{p, q}}\lf|\nabla u(x)\r|^pdx\r)^{\frac{q}{p}}.\nonumber
\end{eqnarray}
By combining inequalities (\ref{eq23}), (\ref{eq24}) and (\ref{eq25}), we obtain 
\begin{equation}\label{eq26}
\lf(\int_{\mathring{U_1}\setminus\bigcup_k\overline{\mathcal A_k}}\lf|\nabla E(u)(x)\r|^qdx\r)^{\frac{1}{q}}\leq C\lf(\int_{\boz_{p, q}}\lf|u(x)\r|^p+\lf|\nabla u(x)\r|^pdx\r)^{\frac{1}{p}}
\end{equation}
By (\ref{eq:norm2}) and the fact that the collection $\{\widetilde{\mathscr C_k}\}$ is pairwise disjoint, we have 
\begin{eqnarray}\label{eq29}
\int_{\bigcup_k\mathscr A_k}\lf|\nabla E(u)(x)\r|^qdx&\leq&C\sum_k\int_{\mathscr C_k}\lf|\nabla u(x)\r|^qdx\\
&+&C\sum_k\int_{\mathscr D^L_k}\lf(\sqrt{\frac{1}{x_n^2+\lf(s-\frac{\tilde r_k}{2}\r)}}\r)^q\lf|u\circ\mathscr R(x)\r|^qdx\nonumber\\
&+&C\sum_k\int_{\mathscr D^U_k}\lf(\sqrt{\frac{1}{(1-x_n)^2+\lf(s-\frac{\tilde r_k}{2}\r)^2}}\r)^q\lf|u\circ\mathscr R(x)\r|^qdx\nonumber\\
&+&C\sum_k\int_{\mathscr A_k\setminus\lf(\overline{\mathscr D^L_k}\cup\overline{\mathscr D^U_k}\r)}\lf(\frac{1}{\tilde r_k}\r)^q\lf|u\circ\mathscr R(x)\r|^qdx.\nonumber
\end{eqnarray}
Since the collection $\{\mathscr C_k\}$ is pairwise disjoint, the H\"older inequality gives
\begin{equation}\label{eq30}
\sum_k\int_{\mathscr C_k}\lf|\nabla u(x)\r|^qdx=\int_{\bigcup_k\mathscr C_k}\lf|\nabla u(x)\r|^qdx\leq C\lf(\int_{\boz_{p, q}}\lf|\nabla u(x)\r|^pdx\r)^{\frac{q}{p}}.
\end{equation}
The H\"older inequality gives 
\begin{multline}\label{eq31}
\sum_k\int_{\mathscr D^L_k}\lf(\sqrt{\frac{1}{x_n^2+\lf(s-\frac{\tilde r_k}{2}\r)^2}}\r)^q\lf|u\circ\mathscr R(x)\r|^qdx\leq\lf(\sum_k\int_{\mathscr D^L_k}\lf|u\circ\mathscr R(x)\r|^pdx\r)^{\frac{q}{p}}\\
\times\lf(\sum_k\int_{\mathscr D^L_k}\lf(\sqrt{\frac{1}{x_n^2+\lf(s-\frac{\tilde r_k}{2}\r)^2}}\r)^{\frac{pq}{p-q}}\r)^{\frac{p-q}{p}}.
\end{multline}
By (\ref{eq:esti3}), the change of variables formula gives
\begin{equation}\label{eq32}
\sum_k\int_{\mathscr D^L_k}\lf|u\circ\mathscr R(x)\r|^pdx\leq C\sum_k\int_{\mathscr R\lf(\mathscr D^L_k\r)}\lf|u(x)\r|^pdx\leq C\int_{\boz_{p, q}}\lf|u(x)\r|^pdx.
\end{equation}
With 
$$l_k:=\sqrt{{x_n^2+\lf(s-\frac{\tilde r_k}{2}\r)^{\color{red}2}}},$$ 
by (\ref{eq:radius}), we have 
\begin{eqnarray}\label{eq33}
\sum_k\int_{\mathscr D^L_k}\lf(\sqrt{\frac{1}{x_n^2+\lf(s-\frac{\tilde r_k}{2}\r)^2}}\r)^{\frac{pq}{p-q}}&\leq& C\sum_k\tilde r_k\int_0^{\tilde r_k}l_k^{n-2-\frac{pq}{p-q}}dl_k\\
&\leq& C\sum_k\tilde r_k^{n-\frac{pq}{p-q}}\leq C\sum_k\lf(\frac{1}{2^{k+1}}\r)^{n-\frac{pq}{p-q}}<\fz.\nonumber
\end{eqnarray}
By combining (\ref{eq31}), (\ref{eq32}) and (\ref{eq33}), we obtain
\begin{equation}\label{eq34}
\sum_k\int_{\mathscr D^L_k}\lf(\sqrt{\frac{1}{x_n^2+\lf(s-\frac{\tilde r_k}{2}\r)^2}}\r)^q\lf|u\circ\mathscr R(x)\r|^qdx\leq C\lf(\int_{\boz_{p,q}}\lf|u(x)\r|^pdx\r)^{\frac{q}{p}}.
\end{equation}
Similarly, we obtain 
\begin{equation}\label{eq35}
\sum_k\int_{\mathscr D^U_k}\lf(\sqrt{\frac{1}{\lf(1-x_n\r)^2+\lf(s-\frac{\tilde r_k}{2}\r)^2}}\r)^q\lf|u\circ\mathscr R(x)\r|^qdx\leq C\lf(\int_{\boz_{p,q}}\lf|u(x)\r|^pdx\r)^{\frac{q}{p}}.
\end{equation}
Since the collection $\{\widetilde{\mathscr C_k}\}$ is pairwise disjoint, the H\"older inequality implies 
\begin{multline}\label{eq36}
\sum_k\int_{\mathscr A_k\setminus\lf(\overline{\mathscr D^L_k}\cup\overline{\mathscr D^U_k}\r)}\lf(\frac{1}{\tilde r_k}\r)^q\lf|u\circ\mathscr R(x)\r|^qdx\leq\lf(\sum_k\int_{\mathscr A_k\setminus\lf(\overline{\mathscr D^L_k}\cup\overline{\mathscr D^U_k}\r)}\lf|u\circ\mathscr R(x)\r|^pdx\r)^{\frac{q}{p}}\\
 \times\lf(\sum_k\int_{\mathscr A_k\setminus\lf(\overline{\mathscr D^L_k}\cup\overline{\mathscr D^U_k}\r)}\lf(\frac{1}{\tilde r_k}\r)^{\frac{pq}{p-q}}dx\r)^{\frac{p-q}{p}}.                         
\end{multline}
By (\ref{eq:esti3}), the change of variables formula gives
\begin{eqnarray}\label{eq37}
\sum_k\int_{\mathscr A_k\setminus\lf(\overline{\mathscr D^L_k}\cup\overline{\mathscr D^U_k}\r)}\lf|u\circ\mathscr R(x)\r|^pdx&\leq& C\sum_k\int_{\mathscr R\lf(\mathscr A_k\setminus\lf(\overline{\mathscr D^L_k}\cup\overline{\mathscr D^U_k}\r)\r)}\lf|u(x)\r|^pdx\\
        &\leq& C\int_{\boz_{p,q}}\lf|u(x)\r|^pdx.\nonumber
\end{eqnarray}
With a simple computation, we have 
\begin{equation}\label{eq38}
\sum_k\int_{\mathscr A_k\setminus\lf(\overline{\mathscr D^L_k}\cup\overline{\mathscr D^U_k}\r)}\lf(\frac{1}{\tilde r_k}\r)^{\frac{pq}{p-q}}dx\leq C\sum_k\tilde r_k^{n-1-\frac{pq}{p-q}}\leq C\sum_k\lf(2^{-k}\r)^{n-1-\frac{pq}{p-q}}<\fz.
\end{equation}
By combining (\ref{eq36}), (\ref{eq37}) and (\ref{eq38}), we obtain 
\begin{equation}\label{eq39}
\sum_k\int_{\mathscr A_k\setminus\lf(\overline{\mathscr D^L_k}\cup\overline{\mathscr D^U_k}\r)}\lf(\frac{1}{\tilde r_k}\r)^q\lf|u\circ\mathscr R(x)\r|^qdx\leq C\lf(\int_{\boz_{p, q}}\lf|u(x)\r|^pdx\r)^{\frac{q}{p}}.
\end{equation}
By combining (\ref{eq29}), (\ref{eq30}), (\ref{eq34}), (\ref{eq35}) and (\ref{eq39}), we obtain 
\begin{equation}\label{eq40}
\int_{\bigcup_k\mathscr A_k}\lf|\nabla E(u)(x)\r|^qdx\leq C\lf(\int_{\boz_{p, q}}\lf|u(x)\r|^p+\lf|\nabla u(x)\r|^pdx\r)^{\frac{q}{p}}.
\end{equation}
Since $\mathcal A_k$ is a subset of $\widetilde{\mathcal C_k}$, by (\ref{eq:norm4}), we have 
\begin{multline}\label{eq41}
\int_{\bigcup_k\mathcal A_k}\lf|\nabla E(u)(x)\r|^qdx\leq C\sum_k\int_{P_k}\lf|\nabla u(x)\r|^qdx\\
+C\sum_k\int_{\mathcal D^U_k}\lf(\sqrt{\frac{1}{\lf(1-x_n\r)^2+\lf(s-\frac{r_k}{2}\r)^2}}\r)^q\lf(\lf|u\circ R_1(x)\r|^q+\lf|u\circ\mathcal R(x)\r|^q\r)dx\\
+C\sum_k\int_{\mathcal D^L_k}\lf(\sqrt{\frac{1}{x_n^2+\lf(s-\frac{r_k}{2}\r)^2}}\r)^q\lf(\lf|u\circ R_1(x)\r|^q+\lf|u\circ\mathcal R(x)\r|^q\r)dx\\
+C\sum_k\int_{\mathcal A_k\setminus\lf(\overline{\mathcal D^L_k}\cup\overline{\mathcal D^U_k}\r)}\lf(\frac{1}{r_k}\r)^q\lf(\lf|u\circ R_1(x)\r|^q+\lf|u\circ\mathcal R(x)\r|^q\r)dx,
\end{multline}
Since the collection $\{P_k\}$ is pairwise disjoint, the H\"older inequality gives
\begin{equation}\label{eq42}
\sum_k\int_{P_k}\lf|\nabla u(x)\r|^qdx\leq C\lf(\int_{\boz_{p,q}}\lf|\nabla u(x)\r|^pdx\r)^{\frac{q}{p}}.
\end{equation}
The H\"older inequality gives 
\begin{multline}\label{eq43}
\sum_k\int_{\mathcal D^U_k}\lf(\sqrt{\frac{1}{\lf(1-x_n\r)^2+\lf(s-\frac{r_k}{2}\r)^2}}\r)^q\lf(\lf|u\circ R_1(x)\r|^q+\lf|u\circ\mathcal R(x)\r|^q\r)dx\\
\leq\lf(\sum_k\int_{\mathcal D^U_k}\lf(\sqrt{\frac{1}{\lf(1-x_n\r)^2+\lf(s-\frac{r_k}{2}\r)^2}}\r)^{\frac{pq}{p-q}}dx\r)^{\frac{p-q}{p}}\\
\times\lf(\sum_k\int_{\mathcal D_k}\lf(\lf|u\circ R_1(x)\r|^p+\lf|u\circ\mathcal R(x)\r|^p\r)dx\r)^{\frac{q}{p}}.
\end{multline}
By (\ref{eq:esti1}) and (\ref{eq:esti4}), the change of variables formula gives 
\begin{multline}\label{eq44}
\sum_k\int_{\mathcal D_k}\lf(\lf|u\circ R_1(x)\r|^p+\lf|u\circ\mathcal R(x)\r|^p\r)dx\\
\leq C\lf(\int_{R_1\lf(\bigcup_k\mathcal D_k\r)}\lf|u(x)\r|^pdx+\int_{\mathcal R\lf(\bigcup_k\mathcal D_k\r)}\lf|u(x)\r|^pdx\r)\\
\leq C\int_{\boz_{p,q}}\lf|u(x)\r|^pdx.
\end{multline}
With 
\[l_k:=\sqrt{{\lf(1-x_n\r)^2+\lf(s-\frac{r_k}{2}\r)^2}},\]
by (\ref{eq:radius}), we have
\begin{eqnarray}\label{eq45}
\sum_k\int_{\mathcal D^U_k}\lf(\sqrt{\frac{1}{\lf(1-x_n\r)^2+\lf(s-\frac{r_k}{2}\r)^2}}\r)^{\frac{pq}{p-q}}dx&\leq&C\sum_kr_k\int_0^{r_k}l_k^{n-2-\frac{pq}{p-q}}dl_k\\
    &\leq&C\sum_kr_k^{n-\frac{pq}{p-q}}<\fz.\nonumber
\end{eqnarray}
Hence, by combining (\ref{eq43}), (\ref{eq44}) and (\ref{eq45}), we obtain 
\begin{multline}\label{eq46}
\sum_k\int_{\mathcal D^U_k}\lf(\sqrt{\frac{1}{\lf(1-x_n\r)^2+\lf(s-\frac{r_k}{2}\r)^2}}\r)^q\lf(\lf|u\circ R_1(x)\r|^q+\lf|u\circ\mathcal R(x)\r|^q\r)dx\\
\leq C\lf(\int_{\boz_{p, q}}\lf|u(x)\r|^pdx\r)^{\frac{q}{p}}.
\end{multline}
By replacing $(1-x_n)$ by $x_n$ and $\mathcal D^U_k$ by $\mathcal D^L_k$, the same argument gives 
\begin{multline}\label{eq47}
\sum_k\int_{\mathcal D^L_k}\lf(\sqrt{\frac{1}{x_n^2+\lf(s-\frac{r_k}{2}\r)^2}}\r)^{\frac{pq}{p-q}}\lf(\lf|u\circ R_1(x)\r|^q+\lf|u\circ\mathcal R(x)\r|^q\r)dx\\
\leq C\lf(\int_{\boz_{p,q}}\lf|u(x)\r|^pdx\r)^{\frac{q}{p}}.
\end{multline}
The H\"older inequality gives 
\begin{multline}\label{eq48}
\sum_k\int_{\mathcal A_k\setminus\lf(\overline{\mathcal D^L_k}\cup\overline{\mathcal D^U_k}\r)}\lf(\frac{1}{r_k}\r)^q\lf(\lf|u\circ R_1(x)\r|^q+\lf|u\circ\mathcal R(x)\r|^q\r)dx\\
\leq\lf(\sum_k\int_{\mathcal A_k\setminus\lf(\overline{\mathcal D^L_k}\cup\overline{\mathcal D^U_k}\r)}\lf(\lf|u\circ R_1(x)\r|^p+\lf|u\circ\mathcal R(x)\r|^p\r)dx\r)^{\frac{q}{p}}\\
\times\lf(\sum_k\int_{\mathcal A_k\setminus\lf(\overline{\mathcal D^L_k}\cup\overline{\mathcal D^U_k}\r)}\lf(\frac{1}{r_k}\r)^{\frac{pq}{p-q}}dx\r)^{\frac{p-q}{p}}.                
\end{multline}
By (\ref{eq:radius}), a simple computation gives
\begin{equation}\label{eq49}
\sum_k\int_{\mathcal A_k\setminus\lf(\overline{\mathcal D^L_k}\cup\overline{\mathcal D^U_k}\r)}\lf(\frac{1}{r_k}\r)^{\frac{pq}{p-q}}dx\leq C\sum_kr_k^{n-1-\frac{pq}{p-q}}<\fz.
\end{equation}
By (\ref{eq:esti1}) and (\ref{eq:esti4}), the change of variables formula implies
\begin{multline}\label{eq50}
\sum_k\int_{\mathcal A_k\setminus\lf(\overline{\mathcal D^L_k}\cup\overline{\mathcal D^U_k}\r)}\lf(\lf|u\circ R_1(x)\r|^p+\lf|u\circ\mathcal R(x)\r|^p\r)dx\\
\leq C\int_{R_1\lf(\mathcal A_k\setminus\lf(\overline{\mathcal D^L_k}\r)\r)}\lf|u(x)\r|^pdx+\int_{\mathcal R\lf(\mathcal A_k\setminus\lf(\overline{\mathcal D^L_k}\cup\overline{\mathcal D^U_k}\r)\r)}\lf|u(x)\r|^pdx\\
\leq C\int_{\boz_{p, q}}\lf|u(x)\r|^pdx.
\end{multline}
By combining (\ref{eq48}), (\ref{eq49}) and (\ref{eq50}), we obtain 
\begin{equation}\label{eq51}
\sum_k\int_{\mathcal A_k\setminus\lf(\overline{\mathcal D^L_k}\cup\overline{\mathcal D^U_k}\r)}\lf(\frac{1}{r_k}\r)^q\lf(\lf|u\circ R_1(x)\r|^q+\lf|u\circ\mathcal R(x)\r|^q\r)dx\leq C\lf(\int_{\boz_{p, q}}\lf|u(x)\r|^pdx\r)^{\frac{q}{p}}.
\end{equation}
By combining (\ref{eq41}), (\ref{eq42}), (\ref{eq46}), (\ref{eq47}) and (\ref{eq51}), we obtain 
\begin{equation}\label{eq52}
\int_{\bigcup_k\mathcal A_k}\lf|\nabla E(u)(x)\r|^qdx\leq C\lf(\int_{\boz_{p, q}}\lf(\lf|u(x)\r|^p+\lf|\nabla u(x)\r|^p\r)dx\r)^{\frac{q}{p}},
\end{equation}
where the constant $C$ is independent of $u$. Finally, by combining (\ref{eq21}), (\ref{eq22}), (\ref{eq26}), (\ref{eq40}) and (\ref{eq52}), we obtain the desired norm control for the derivative of $E(u)$. That is
\begin{equation}\label{eq53}
\lf(\int_{\mathcal C}\lf|\nabla E(u)(x)\r|^qdx\r)^{\frac{1}{q}}\leq C\lf(\int_{\boz_{p, q}}\lf(\lf|u(x)\r|^p+\lf|\nabla u(x)\r|^p\r)dx\r)^{\frac{1}{p}}
\end{equation}
with a constant $C$ independent of $u$.

By combining (\ref{eq20}) and (\ref{eq53}), we know that the extension operator $E$ defined in (\ref{eq:extension}) is bounded from the dense subspace $C^\fz(\overline{\boz_{p,q}})\cap W^{1, p}(\boz_{p, q})$ of $W^{1, p}(\boz_{p, q})$ to $W^{1, q}(\mathcal C)$. This means that for every $u\in C^\fz(\overline{\boz_{p,q}})\cap W^{1, p}(\boz_{p,q})$, we have $E(u)\in W^{1, q}(\mathcal C)$ with 
\[\|E(u)\|_{W^{1, q}(\mathcal C)}\leq C\|u\|_{W^{1, p}(\boz_{p, q})},\]
where the constant $C$ is independent of $u$. Since $C^{\fz}(\overline{\boz_{p,q}})\cap W^{1, p}(\boz_{p, q})$ is dense in $W^{1, p}(\boz_{p, q})$, for every $u\in W^{1, p}(\boz_{p, q})$, there exists a Cauchy sequence $\{u_m\}\subset C^\fz(\overline{\boz_{p,q}})\cap W^{1, p}(\boz_{p,q})$ which converges to $u$ in $W^{1, p}(\boz_{p,q})$. Furthermore, there exists a subsequence of $\{u_m\}$ which converges to $u$ almost everywhere in $\boz_{p, q}$. To simplify the notation, we still use $\{u_m\}$ for this subsequence. By (\ref{eq53}), $\{E(u_m)\}$ is a Cauchy sequence in the Sobolev space $W^{1, q}(\mathcal C)$ and hence converges to some function $v\in W^{1, q}(\mathcal C)$ with respect to $W^{1, q}$-norm. Furthermore, there exists a subsequence of $\{E(u_m)\}$ which converges to $v$ almost everywhere in $\mathcal C$. On the other hand, by the definition of $E(u_m)$ and $E(u)$, we have 
\[\lim_{m\to\fz}E(u_m)(x)=E(u)(x)\]
for almost every $x\in\mathcal C$. Hence $v(x)=E(u)(x)$ almost everywhere. This implies that $E(u)\in W^{1, q}(\mathcal C)$ with 
\begin{eqnarray}\label{eq54}
\|E(u)\|_{W^{1, q}(\mathcal C)}&=&\|v\|_{W^{1, q}(\mathcal C)}=\lim_{m\to\fz}\|E(u_m)\|_{W^{1, q}(\mathcal C)}\\
&\leq&C\lim_{m\to\fz}\|u\|_{W^{1, p}(\boz_{p,q})}\leq C\|u\|_{W^{1,p}(\boz_{p,q})},\nonumber
\end{eqnarray}
where the constant $C$ is independent of $u$. Hence, the linear extension operator $E$ defined in (\ref{eq:extension}) is bounded from $W^{1,p}(\boz_{p,q})$ to $W^{1, q}(\mathcal C)$. The fact that $\mathcal C$ is a Sobolev $(q, q)$-extension domain implies $\boz_{p,q}$ is a Sobolev $(p, q)$-extension domain.
\end{proof}

Next, we show that $\boz_{p, q}$ is not a $(L^{1, p}, L^{1, q})$-extension domain for $1\leq q<n-1$ and $(n-1)p/(n-1-q)<p<\fz$.
\begin{thm}\label{th:nothomo}
 Fix $1\leq q<n-1$ and $(n-1)q/(n-1-q)<p<\fz$. Then $\boz_{p, q}\subset\rn$ defined in (\ref{eq:boz}) is not an $(L^{1, p}, L^{1, q})$-extension domain.
\end{thm}
\begin{proof}
Fix $1\leq q<n-1$ and $(n-1)q/(n-1-q)<p<\fz$. Assume that $\boz_{p, q}$ is a $(L^{1, p}, L^{1, q})$-extension domain, then for every $u\in L^{1, p}(\boz_{p, q})$, there exists a function $E(u)\in L^{1, q}(B(0, 10))$ with $E(u)\big|_{\boz_{p,q}}\equiv u$ and 
\[\|\nabla E(u)\|_{L^q(B(0, 10))}\leq C\|\nabla u\|_{L^p(\boz_{p, q})}\]
with a positive constant $C$ independent of $u$. By Theorem \ref{th:poincare}, we have 
$$L^{1, q}(B(0, 10))=W^{1, q}(B(0, 10))$$
as a set of functions. Hence we have $E(u)\in W^{1, q}(B(0, 10))$ and $u\in L^{1, q}(\boz_{p, q})$. Hence, to prove the theorem, it suffices to construct a function $u$ which is in the homogeneous Sobolev space $L^{1, p}(\boz_{p, q})$ but not in the Sobolev space $W^{1, q}(\boz_{p, q})$.  Set $u\equiv 0$ on $\mathcal Q_o$ and $u\equiv (4^k)^{\frac{n-1}{q}}$ on $\mathscr C_k$. On every $\mathcal C_k$, we define $u$ to be increasing linearly  with respect to $x_n$ from $0$ in $\mathcal Q_0$ to $(4^k)^{\frac{n-1}{q}}$ in $\mathscr C_k$. To be precise, 
\[u(x)=(4^k)^{\frac{n-1}{q}}x_n-(4^k)^{\frac{n-1}{q}}\]
for every $x=(x_1,x_2,\cdots, x_n)\in\mathcal  C_k$. By some simple computations, we have 
\begin{equation}\label{eq:derivative}
|\nabla u(x)|=\begin{cases}
(4^k)^{\frac{n-1}{q}},\ &\ {\rm if}\ x\in \mathcal C_k, \\
0,\ &\ {\rm elsewhere}.
\end{cases}
\end{equation}
Hence, we have 
\begin{eqnarray}
\int_{\boz_{p,q}}|\nabla u(x)|^pdx&\leq&\sum_{k=1}^\fz(4^k)^{\frac{(n-1)p}{q}}\mathcal H^n\lf(\mathcal C_k\r)\nonumber\\
&\leq&C\sum_{k=1}^\fz(4^k)^{\frac{(n-1)p}{q}}\lf(\frac{1}{4^{k+1}}\r)^{1+\frac{(n-1)p}{q}}\nonumber\\
&\leq&C\sum_{k=1}^\fz\frac{1}{4^{k+1}}<\fz.\nonumber
\end{eqnarray}
Hence $u\in L^{1, p}(\boz_{p, q})$. By a simple computation, we have 
\begin{eqnarray}
\int_{\boz_{p, q}}|u(x)|^qdx&\geq&\sum_{k=1}^\fz\int_{\mathcal Q_k\cap\boz_{p, q}}|u(x)|^qdx\nonumber\\
&\geq&\sum_{k=1}^\fz\lf(\frac{1}{4^{k+1}}\r)^{n-1}(4^k)^{n-1}=\fz.\nonumber
\end{eqnarray}
Hence, we also have $u\notin W^{1, q}(\boz_{p, q})$ as desired.
\end{proof}
\section{Construction of an Example}\label{difdomain}
In this section, we construct a domain $\boz\subset\rn$ with $n\geq 3$ which shows that $\frac{(n-1)q}{(n-1-q)}<p< \infty$ is a sharp requirement above.\\

Let $\mathcal Q_0:=[0, 20]\times[0, 1]^{n-1}\subset\rn$ be a closed rectangle in the Euclidean space $\rn$. Let $\mathscr Q_0:=[0, 20]\times[0, 1]^{n-2}\times\{0\}$ be the lower $(n-1)$-dimensional surface of the rectangle $\mathcal Q_0$. We define a sequence of points $\{z_k\in\mathscr Q_0\}_{k=1}^\fz$ by setting 
\[z_1:=(1, 0, \cdots, 0)\ {\rm and}\ z_k:=\lf(1+30\sum_{j=2}^k 2^{-j}, 0, \cdots, 0\r)\]
for every $2\leq k<\fz$.

For every point $x=(x_1, x_2,\cdots, x_n)\in\rn$, define $\check{x}:=(x_1, x_2, \cdots, x_{n-1})\in\rr^{n-1}$ to be the projection of $x$ onto the $(n-1)$-dimensional hyperplane $\rr^{n-1}\times\{0\}$. 
We define a cylinder $\mathcal{C}_k$ which by setting
\[\mathcal C_k:=\lf\{x\in\rn: 0\leq|\check x-z_k|\leq 2^{-k-1}, -1\leq x_n\leq 0\r\}.\]
Similarly, we define $\mathcal{C}_{k,\frac{1}{2}}$ by setting
\[\mathcal C_{k,\frac{1}{2}}:=\lf\{x\in\rn: 0\leq|\check x-z_k|\leq 2^{-k-1}, -1\leq x_n\leq -\frac{1}{2}\r\}.\]
Then we define the domain $\boz$ to be the interior of the set 
\[F:=\mathcal Q_0\cup\bigcup_{k=1}^\fz\mathcal{C}_{k}.\]

We define $u_k$ on $\boz$ by setting
\begin{equation}
    u_k(x):= \begin{cases}
        0, \  &{\rm if}\ \ x\in \Omega\setminus \mathcal{C}_{k},\\
        1, \ &{\rm if}\ \ x\in \mathcal{C}_{k,\frac{1}{2}},\\
        2x_n,  \  &{\rm elsewhere}.
    \end{cases}
\end{equation}
Let $1\leq p<\infty$ and $1\leq q<n-1.$ Then,
\begin{equation}\label{spq}
    \|u_k\|_{W^{1,p}(\Omega)}\leq C_n \lf(2^{\frac{-(k+2)(n-1)}{p}}\r).
\end{equation}
Let $T_t$ for $t\in \mathbb{R}$ be any hyperplane perpendicular to $\mathcal{C}_{k}$ that intersect the $\mathcal{C}_{k,\frac{1}{2}}$ and let $Eu_{k}$ be an extension of $u_k.$
 By the Sobolev embedding inequality, we have for a.e. $t\in \mathbb{R}$
\begin{equation}\label{peq}
    2^{\frac{-(k+2)(n-1)}{q_{n-1}^*}}\leq C_{n,q}\left(\int_{T_t}|\nabla Eu_k(x)|^qdx\right)^{\frac{1}{q}}
\end{equation}
where $q_{n-1}^*= \frac{q(n-1)}{(n-1-q)}.$\\
By using \eqref{peq}, we get that
\begin{equation}\label{sppeq}
    2^{\frac{-(k+2)(n-1)}{q_{n-1}^*}}\leq C_{n,q}\|Eu_k\|_{W^{1,q}(\mathbb{R}^n)}.
\end{equation}
Combining \eqref{spq} and \eqref{sppeq}, we conclude that
\begin{eqnarray}
    \frac{\|Eu_k\|_{W^{1,q}(\mathbb{R}^n)}}{\|u_k\|_{W^{1,p}(\Omega)}}&\geq& C_{p,q,n}2^{-(k+2)\left(\frac{(n-1)}{q_{n-1}^\star}-\frac{(n-1)}{p}\right)}.
\end{eqnarray}
For $p<\frac{q(n-1)}{(n-1-q)}$,
\begin{equation*}
    \frac{\|Eu_k\|_{W^{1,q}(\mathbb{R}^n)}}{\|u_k\|_{W^{1,p}}(\Omega)}\to \infty \hspace{2mm} \text{as} \hspace{2mm} k\to \infty.
\end{equation*}
Therefore, a $(W^{1,p},W^{1,q})$-extension is not possible unless $p\geq \frac{q(n-1)}{(n-1-q)}$ where $1\leq q<n-1.$ One can check also that no $(W^{1,p},W^{1,q})$-extension possible also when $q\geq n-1.$ 

 
\section{Example shown in the main proof is strong}
Let $\Omega$ be the domain as constructed in \eqref{eq:boz}. For $k\in \mathbb{N}$, we define $u$ on $\boz$ by setting
\begin{equation}
    u_k(x):= \begin{cases}
        0, \  &{\rm if}\ \ x\in \Omega\setminus \mathcal{C}_{k}\cup \mathscr C_k,\\
        1, \ &{\rm if}\ \ x\in \mathscr C_k,\\
        x_n-2,  \  &{\rm elsewhere}.
    \end{cases}
\end{equation}
Now, 
\begin{equation}
    \|u_k\|_{W^{1,p}(\Omega)}= \|u_k\|_{L^p{(\Omega)}}+\|\nabla u_k\|_{L^p(\Omega)},
\end{equation}
which implies that
\begin{eqnarray}
     \|u_k\|_{W^{1,p}(\Omega)} &\leq & C\left[\left(\frac{1}{2^{k+1}}\right)^{\frac{n-1}{p}}+ \left(\frac{1}{4^{k}}\right)^{\left(\frac{1}{(n-1)}+\frac{p}{q}\right)\frac{(n-1)}{p}}\right]
\end{eqnarray}
Since $\frac{1}{p}+\frac{n-1}{q}$ is positive, we get that
\begin{equation}\label{eqp}
         \|u_k\|_{W^{1,p}(\Omega)} \leq C\left[\left(\frac{1}{2^{k+1}}\right)^{\frac{n-1}{p}}\right]
\end{equation}
Let $T_t$ for $t\in \mathbb{R}$ be any hyperplane perpendicular to $\mathcal{C}_{k}$ that intersect and let $Eu_{k}$ be extension of $u_k$.
 By Sobolev embedding inequality, we have for a.e. $t\in \mathbb{R}$
\begin{equation}\label{eqpe}
    2^{\frac{-(k+2)(n-1)}{q_{n-1}^*}}\leq C_{n,q}\left(\int_{T_t}|\nabla Eu_k(x)|^qdx\right)^{\frac{1}{q}}
\end{equation}
where $q_{n-1}^*= \frac{q(n-1)}{(n-1-q)}.$\\
By using \eqref{eqpe}, we get that
\begin{equation}\label{eqpp}
   2^{\frac{-(k+2)(n-1)}{q_{n-1}^*}}\leq C_{n,q}\leq \|Eu_k\|_{W^{1,q}(\mathbb{R}^n)}. 
\end{equation}
 Combining \eqref{eqp} and \eqref{eqpp}, we get that
\begin{eqnarray}\label{meq}
\frac{\|Eu_{k}\|_{W^{1,q}(\mathbb{R}^n)}}{\|u_{k}\|_{W^{1,p}(\Omega)}}&\geq& C_{p,q,n}2^{-(k+2)\left(\frac{(n-1)}{q_{n-1}^\star}-\frac{(n-1)}{p}\right)}
\end{eqnarray}
For $p< q^\star_{n-1}$, \eqref{meq} tends to infinity as $k$ tends to infinity.
\begin{remark}
    If we consider the domain $\widetilde{\Omega}$ a slightly modification of the domain as constructed in \eqref{difdomain}. We will fix the cube $\mathcal{Q}_0$ and attach cylinders $\mathcal{C}_k$ to it such the the ratio of width and height of the cylinder goes to zero as $k$ goes to infinity. Then $\widetilde{\Omega}$ is not a $(W^{1,p}, W^{1,q})$ extension domain if $1\leq q <n-1$ and $p< \frac{q(n-1)}{(n-1-q)}$.  
\end{remark}

\end{document}